\documentclass[a4paper, reqno, 11pt]{amsart}
\usepackage{amsmath, amsthm, amssymb, geometry, latexsym, mathrsfs}
\usepackage{multicol}
\usepackage[all]{xy}
\xyoption{poly}
\usepackage{tabularx}
\usepackage{booktabs}
\usepackage{caption}
\makeatletter
\@namedef{subjclassname@2020}{%
  \textup{2020} Mathematics Subject Classification}
\makeatother

\def\al{\alpha}
\def\be{\beta}
\def\ga{\gamma}
\def\Ga{\Gamma}

\def\De{\Delta}
\def\ep{\varepsilon}
\def\e{\varepsilon}

\def\th{\theta}

\def\la{\lambda}
\def\La{\Lambda}

\def\si{\sigma}

\def\vph{\varphi}

\def\om{\omega}
\def\Om{\Omega}

\def\cA{{\mathcal A}}

\def\cE{{\mathcal E}}

\def\cK{{\mathcal K}}

\def\cM{{\mathcal M}}
\def\cN{{\mathcal N}}
\def\cO{{\mathcal O}}

\def\cS{{\mathcal S}}

\def\cU{{\mathcal U}}

\def\cX{{\mathcal X}}

\def\sC{{\mathscr C}}

\def\sE{{\mathscr E}}
\def\sT{{\mathscr T}}

\def\ang#1{{\langle #1 \rangle}}

\def\P{{\mathbb P}}

\def\Z{{\mathbb Z}}
\def\N{{\mathbb N}}
\def\F{{\mathbb F}}
\def\M{{\mathbb M}}

\def\Aut{\operatorname{Aut}}

\def\fchar{\operatorname{char}}

\def\Coker{\operatorname{Coker}}
\def\depth{\operatorname{depth}}

\def\dim{\operatorname{dim}}

\def\End{\operatorname{End}}
\def\Ext{\operatorname{Ext}}
\def\en{\operatorname{e}}

\def\gldim{\operatorname{gldim}}

\def\Hom{\operatorname{Hom}}

\def\HH{\operatorname{HH}}

\def\Im{\operatorname{Im}}

\def\injdim{\operatorname{injdim}}

\def\Ker{\operatorname{Ker}}

\def\mnull{{\operatorname{null}}}

\def\op{\operatorname{o}}
\def\projdim{\operatorname{projdim}}

\def\Proj{\operatorname{Proj}}

\def\rank{\operatorname{rank}}

\def\sup{\operatorname{sup}}

\def\RHom{\operatorname{RHom}}

\newcommand{\Lotimes}{\otimes^{\operatorname L}}

\def\grmod{\operatorname{\mathsf{grmod}}}
\def\GrMod{\operatorname{\mathsf{GrMod}}}

\def\perf{\operatorname{\mathsf{perf}}}
\def\CM{\operatorname{\mathsf{CM}}^{\Z}}

\def\qgr{\operatorname{\mathsf{qgr}}}

\def\coh{\operatorname{\mathsf{coh}}}

\def\uCM{\operatorname{\underline{\mathsf{CM}}}^{\Z}}

\def\fdim{\operatorname{\mathsf{fdim}}}

\def\D{\mathsf{D}}
\def\Db{\mathsf{D^b}}
\def\mod{\operatorname{\mathsf{mod}}}

\def\<{\langle}
\def\>{\rangle}
\def\rnum#1{\expandafter{\romannumeral #1}}
\def\Rnum#1{\uppercase\expandafter{\romannumeral #1}}

\theoremstyle{plain} 
\newtheorem{thm}{Theorem}[section]
\newtheorem{cor}[thm]{Corollary}
\newtheorem*{thm*}{Theorem}
\newtheorem*{cor*}{Corollary}
\newtheorem{lem}[thm]{Lemma}
\newtheorem{prop}[thm]{Proposition}

\theoremstyle{definition}
\newtheorem{dfn}[thm]{Definition}
\newtheorem{ex}[thm]{Example}

\newtheorem{ntn}[thm]{Notation}
\newtheorem{rem}[thm]{Remark}

\newcommand{\thmref}[1]{Theorem~\ref{#1}}
\newcommand{\lemref}[1]{Lemma~\ref{#1}}

\newcommand{\propref}[1]{Proposition~\ref{#1}}
\newcommand{\exref}[1]{Example~\ref{#1}}
\newcommand{\dfnref}[1]{Definition~\ref{#1}}

\numberwithin{equation}{section}

\begin{document}

\title{Derived categories of skew quadric hypersurfaces}
\author{Kenta Ueyama}
\address{
Department of Mathematics, 
Faculty of Education,
Hirosaki University, 
1 Bunkyocho, Hirosaki, Aomori 036-8560, Japan}
\email{k-ueyama@hirosaki-u.ac.jp} 
\thanks{The author was supported by JSPS Grant-in-Aid for Early-Career Scientists 18K13381.}


\begin{abstract}
The existence of a full strong exceptional sequence in the derived category of a smooth quadric hypersurface was proved by Kapranov.
In this paper, we present a skew generalization of this result.
Namely, we show that if $S$ is a standard graded $(\pm 1)$-skew polynomial algebra in $n$ variables with $n \geq 3$ and
$f = x_1^2+\cdots +x_n^2 \in S$, then the derived category $\operatorname{\mathsf{D^b}}(\operatorname{\mathsf{qgr}} S/(f))$
of the noncommutative scheme $\operatorname{\mathsf{qgr}} S/(f)$
has a full strong exceptional sequence.
The length of this sequence is given by $n-2+2^r$ where $r$ is the nullity of a certain matrix over $\mathbb F_2$.
As an application, by studying the endomorphism algebra of this sequence,
we obtain the classification of $\operatorname{\mathsf{D^b}}(\operatorname{\mathsf{qgr}} S/(f))$ for $n=3, 4$.
\end{abstract}

\maketitle

\tableofcontents

\section{Introduction}
Throughout this paper, $k$ denotes an algebraically closed field of $\fchar k \neq 2$.

The theory of exceptional sequences plays an important role in the study of triangulated categories.
It is very useful to describe the structure of a triangulated category.
For example, it is well-known that there is a full strong exceptional sequence 
\[ \Db(\coh \P^{n-1}) = \<\cO_{\P^{n-1}}(-n+1),\dots, \cO_{\P^{n-1}}(-1),\cO_{\P^{n-1}} \>.\]
Note that a noncommutative generalization (more precisely, an AS-regular version) of this result is also known
(see \cite[Propositions 4.3 and 4.4]{MM}).

As another interesting case, M. M. Kapranov showed the existence of a full strong exceptional sequence in the derived category of a smooth quadric hypersurface.

\begin{thm}[{\cite{K1}, \cite[Section 4]{K2}}] \label{thm.K}
Let $X \subset \P^{n-1}$ be a smooth quadric hypersurface with $n\geq 3$. Then $\Db(\coh X)$ has a full strong exceptional sequence of the form
\[
\Db(\coh X)=
\begin{cases}
\< \cO_X(-n+3),\dots, \cO_X(-1), \cO_X, \cS \> & \text{if $n$ is odd,} \\
\< \cO_X(-n+3),\dots, \cO_X(-1), \cO_X, \cS_{+}, \cS_{-}  \> & \text{if $n$ is even}
\end{cases}
\]
where $\cS$ and $\cS_{\pm}$ are Arithmetically Cohen-Macaulay bundles, called the Spinor bundles.
\end{thm}
In this paper, we give a skew generalization of this theorem.

\begin{ntn}\label{nota}
For a symmetric matrix $\e:= (\e_{ij}) \in M_n(k)$ such that $\e_{ii}=1$ and $\e_{ij}=\e_{ji} \in \{1, -1\}$ for $i\neq j$, we fix the following notations: 
\begin{enumerate}
\item $S_{\e}$ is the graded skew polynomial algebra $k\langle x_1, \dots, x_n\rangle /(x_ix_j-\e_{ij}x_jx_i)$ where all variables $x_i$ are of degree $1$ (it is called a \emph{standard graded ($\pm 1$)-skew polynomial algebra});
\item $f_{\e}$ is the central element $x_1^2+\cdots +x_n^2\in S_{\e}$;
\item $A_{\e}:=S_{\e}/(f_{\e})$;
\item $G_{\e}$ is the graph with vertex set
$V(G_{\e})=\{1, \dots , n\}$ and edge set $E(G_{\e})=\{ij \mid \e_{ij}=\e_{ji}=1, i \neq j \}$;
\item $\De_{\e}$ is the matrix
\[
\begin{pmatrix}  
 & & &1 \\
 &M(G_\e) & &\vdots \\
 & & &1 \\
1 &\cdots &1 &0
\end{pmatrix}
\in M_{n+1}(\F_2)
\]
where $\F_2$ is the field with two elements $0$ and $1$, and $M(G_\e)$ is the adjacency matrix of $G_\e$ over $\F_2$;
\item $\mnull_{\F_2} \De_{\e}$ is the nullity of the matrix of $\De_{\e}$ over $\F_2$.
\end{enumerate} 
\end{ntn}

In \cite{HU}, Higashitani and the author proved the following structure theorem
for the stable category $\uCM(A_\e)$ of graded maximal Cohen-Macaulay modules over $A_\e$ by combinatorial methods. 

\begin{thm}[{\cite[Theorem 1.3]{HU}}] \label{thm.HU}
We use Notation \ref{nota}. Let $r=\mnull_{\F_2}\De_{\e}$ and $\al=2^r$.
\begin{enumerate}
\item There exists an equivalence of triangulated categories $\uCM(A_\e) \cong \Db(\mod k^\al)$.
\item $A_\e$ has $\al$ indecomposable non-projective graded maximal Cohen-Macaulay modules up to isomorphism and degree shift.
\end{enumerate}
\end{thm}
Note that four graphical operations, called
mutation, relative mutation, Kn\"orrer reduction, and two points reduction,
are crucial for the proof of \thmref{thm.HU} (see \cite[Section 3]{HU}).

Let $X_1,\dots, X_\al \in \CM(A_\e)$ be the complete representatives
of the isomorphism classes of indecomposable non-projective graded maximal Cohen-Macaulay modules
where we assume that every $X_i$ is equipped with the grading $(X_i)_{<0}=0$ and $(X_i)_{0}\neq 0$.

Let $\grmod A_\e$ be the category of finitely generated right graded $A_\e$-modules,
and let $\qgr A_\e = \grmod A_\e/ \fdim A_\e$ be its quotient category by the category $\fdim A_\e$ of finite-dimensional modules.
The category $\qgr A_\e$ plays the role of the category of coherent sheaves on the projective scheme associated with $A_\e$.
The object in $\qgr A_\e$ that corresponds to a module 
$M \in \grmod A_\e$ is denoted by calligraphic letter $\cM$.

As a consequence of \thmref{thm.HU}, it was shown that $\qgr A_\e$ has a nice property.

\begin{thm}[{\cite[Theorem 1.4]{HU}}]
\label{thm.higuey}
Let $A_\e$ be as in Notation \ref{nota}.
Then $\gldim(\qgr A_\e)<\infty$, that is, $\qgr A_\e$ is smooth in the sense of Smith and Van den Bergh \cite{SV}.
\end{thm}

The main result of this paper is the following theorem. 

\begin{thm}[\thmref{thm.main}] \label{thm.main.i}
We use Notation \ref{nota} with $n\geq 3$.
Let $r=\mnull_{\F_2} \De_{\e}$ and $\al=2^r$.
Then $\Db(\qgr A_\e)$ has a full strong exceptional sequence of the form
\[
\Db(\qgr A_\e)=
\< \cA_\e(-n+3), \dots, \cA_\e(-1), \cA_\e, \cX_1, \cX_2, \dots, \cX_\al \>.
\]
Moreover, if $\cU= \cA_\e(-n+3) \oplus \cdots \oplus \cA_\e \oplus \cX_1 \oplus \cdots \oplus \cX_\al$,
then the following statements hold.
\begin{enumerate}
\item $\La_\e := \End_{\qgr A_\e}{\cU}$ is an extremely Fano algebra of $\gldim \La_\e = n-2$ in the sense of Minamoto \cite{Min}.
\item There exists an equivalence of triangulated categories $\Db(\qgr A_\e) \cong \Db(\mod \La_\e)$.
\item The $(n-1)$-preprojective algebra $\Pi \La_\e$ of $\La_\e$ is a right noetherian graded Calabi-Yau algebra of $\gldim \Pi\La_\e = n-1$
(in particular, an AS-regular algebra over $\La_\e$ in the sense of Minamoto and Mori \cite{MM}).
\item There exists an equivalence $\qgr A_\e \cong \qgr \Pi\La_\e$.
\end{enumerate}
\end{thm}

Note that extremely Fano algebras of global dimension $d$
are closely related to $d$-representation-infinite algebras
which play a key role in higher Auslander-Reiten theory \cite{HIO}.
In \thmref{thm.quiv}, we describe the quiver and relations of $\La_\e$.

Now consider the case $\e_{ij}=1$ for all $1\leq i<j \leq n$ so that $A_\e$ is commutative.
It is well-known that $\qgr A_\e \cong \coh X$ for any smooth quadric hypersurface $X$ in $\P^{n-1}$.  
Besides, we have
\[ \mnull_{\F_2} \De_{\e}
=\begin{cases}
0 & \text{if $n$ is odd,}\\
1 & \text{if $n$ is even,}
\end{cases}
\]
so we can recover \thmref{thm.K} from \thmref{thm.main.i}.

Roughly speaking, \thmref{thm.main.i} states that $\cU$ is a nice tilting object in $\Db(\qgr A_\e)$, so this paper makes a contribution to tilting theory.
Since tilting theory often enables us to construct an equivalence between a given triangulated category and the derived category of modules over the endomorphism algebra of a tilting object, it is now indispensable for the study of triangulated categories.
For example, it is well-known that $\Db(\coh X)$ has a tilting object if $X$ is a Grassmannian \cite{K2}, \cite{BLV}, or a rational surface \cite{HP}. 
Furthermore, it is well-known that $\uCM(A)$ has a tilting object if $A$ is the trivial extension algebra $\Delta \La$ of a finite-dimensional algebra $\La$ of finite global dimension \cite{H}, or a graded simple surface singularity \cite{GL1}, \cite{GL2}, \cite{KST}.
If $A$ is a noetherian graded (not necessary commutative) Gorenstein algebra,
then there are strong connections between tilting theories of $\Db(\qgr A)$ and $\uCM(A)$.
For example, Orlov's theorem \cite{O}, \cite{IY} (see \thmref{thm.O} for a particular case)
is a powerful result which gives an embedding between $\Db(\qgr A)$ and $\uCM(A)$, and leads some interesting applications to tilting theory
(see \cite{HIMO}, \cite{BIY},  \cite{MUs}, \cite{T}).
Orlov's theorem also plays an important role in this paper.

If $n=2$, then one can verify that $\Db(\qgr A_\e)$ is equivalent to $\Db(\mod k^2)$.
From our results it follows easily that if $n=3$, then $\Db(\qgr A_\e)$ is equivalent to $\Db(\mod kQ)$
where $Q$ is an extended Dynkin quiver of type $\widetilde{A_1}$ or type $\widetilde{D_4}$ (\exref{ex.n3}).
In the last section (Section 4), as an application of our results, we give the classification of $\Db(\qgr A_\e)$ with $n=4$
via Hochschild cohomology of $\La_\e$.

The author hopes that our results will inspire more work which lies at the intersection of noncommutative algebraic geometry, representation theory of algebras, and combinatorics.

\section{Preliminaries}

\subsection{Basic Notation}
In this paper, a graded algebra means an $\N$-graded algebra over $k$ unless otherwise stated.
Recall that a graded algebra $A=\bigoplus_{i \in \N} A_i$ is \emph{connected} if $A_0 =k$, and it is \emph{locally finite} if $\dim_k A_i < \infty$ for all $i \in \N$.
Note that a right noetherian connected graded algebra is locally finite.

For a graded algebra $A$, we denote by $\GrMod A$ the category of graded right $A$-modules with degree preserving $A$-module homomorphisms,
and by $\grmod A$ the full subcategory consisting of finitely generated graded $A$-modules.
Note that if $A$ is right noetherian, then $\grmod A$ is an abelian category.
We denote by $A^{\op}$ the opposite algebra of $A$.
The category of graded left $A$-modules is identified with $\GrMod A^{\op}$.
We denote by $A^{\en}$ the enveloping algebra of $A$.
The category of graded  $A$-$A$ bimodules is identified with $\GrMod A^{\en}$.

For a graded module $M \in \GrMod A$ and an integer $s\in \Z$,
we define  the \emph{truncation} $M_{\geq s} := \bigoplus_{i\geq s} M_i$
and the \emph{internal shift} $M(s)$ by $M(s)_i = M_{s+i}$.
Note that the rule $M \mapsto M(s)$ is a $k$-linear autoequivalence for $\GrMod A$ and $\grmod A$, called the \emph{internal shift functor}.
For $M, N\in \GrMod A$, we write $\Ext^i_{\GrMod A}(M, N)$ for the extension group in $\GrMod A$, and define
\[\Ext^i_A(M, N):=\bigoplus _{s \in \Z}\Ext^i_{\GrMod A}(M, N(s)).\]

Let $A$ be a right noetherian locally finite $\N$-graded algebra.
We denote by $\fdim A$ the full subcategory of $\grmod A$ consisting of finite-dimensional modules over $k$.
Then the Serre quotient category
\[ \qgr A := \grmod A/\fdim A.\]
is an abelian category.
If $A$ is a commutative graded algebra finitely generated in degree 1 over $k$, then $\qgr A$ is equivalent to
the category $\coh(\Proj A)$ of coherent sheaves on the projective scheme $\Proj A$ by Serre's theorem. 
For this reason, $\qgr A$ is called the \emph{noncommutative projective scheme} associated with $A$.
The study of noncommutative projective schemes has been one of the major themes in noncommutative algebraic geometry
 (see \cite{AZ} for basic information about noncommutative projective schemes).

Let $\pi : \grmod A \to \qgr A$ be the (exact) quotient functor.
The objects of $\qgr A$ will be denoted in calligraphic font (e.g., $\cM = \pi M$, $\cA=\pi A$, $\cX=\pi X$).
Note that the $k$-linear autoequivalence $M \mapsto M(s)$ for $\grmod A$ induces
a $k$-linear autoequivalence $\cM\mapsto \cM(s)$ for $\qgr A$, again called the \emph{internal shift functor}.
For $\cM, \cN\in \qgr A$, we write $\Ext^i_{\qgr A}(M, N)$ for the extension group in $\qgr A$, and define
\[\Ext^i_{\cA}(\cM, \cN):=\bigoplus _{s \in \Z} \Ext^i_{\qgr A}(\cM, \cN(s)).\]
The \emph{global dimension of $\qgr A$} is defined by
\[ \gldim (\qgr A) := \sup\{i \mid \Ext^i_{\qgr A}(\cM, \cN) \neq 0\; \text{for some}\; \cM, \cN \in \qgr A \}.\]
It is easy to see that if $\gldim A<\infty$, then $\gldim(\qgr A)<\infty$.
The condition $\gldim (\qgr A)<\infty$ is considered as the smoothness of $\qgr A$ and the noncommutative graded isolated singularity property of $A$
(see \cite{SV}, \cite{Uis}, \cite{Ucm}).

For example, if $A_\e$ is as in Notation \ref{nota}, then $\gldim(\qgr A_\e)<\infty$ by \cite[Theorem 1.4]{HU}.

\begin{dfn}
A noetherian connected graded algebra $A$ is called
\emph{AS-regular (resp.\, AS-Gorenstein) of dimension $d$ and Gorenstein parameter $\ell$} if
\begin{itemize}
\item{} $\gldim A = d <\infty$ (resp.\, $\injdim_A A = \injdim_{A^{\op}} A= d <\infty$), and
\item{} $\Ext^i_A(k ,A) \cong \Ext^i_{A^{\op}}(k, A) \cong
\begin{cases}
k(\ell) & \text { if } i=d,\\
0 & \text { if } i\neq d.
\end{cases}$
\end{itemize}
\end{dfn}

For example, if $S_\e$ and $A_\e$ are as in Notation \ref{nota}, 
then $S_\e$ is a noetherian AS-regular algebra of dimension $n$ and Gorenstein parameter $n$
and $A_\e$ is a noetherian AS-Gorenstein algebra of dimension $n-1$ and Gorenstein parameter $n-2$.

Let $A$ be a noetherian AS-Gorenstein algebra.
We call $M \in \grmod A$ \emph{graded maximal Cohen-Macaulay} if $\Ext^i_A(M ,A) =0$ for all $i \neq 0$.
We write $\CM(A)$ for the full subcategory of $\grmod A$ consisting of graded maximal Cohen-Macaulay modules.
The \emph{stable category of graded maximal Cohen-Macaulay modules}, denoted by $\uCM(A)$, has the same objects as $\CM(A)$
and the morphism space
\[ \Hom_{\uCM(A)}(M, N) = \Hom_{\CM(A)}(M,N)/P(M,N) \]
where $P(M,N)$ consists of degree preserving $A$-module homomorphisms factoring through a graded projective module.

A \emph{triangulated category} is an additive category $\sT$ with an autoequivalence
$[1]: \sT \to \sT$ (called \emph{translation functor}) and a class of sequences $X \to Y \to Z \to X[1]$
(called \emph{distinguished triangles}), satisfying certain axioms.
For example, the bounded derived category $\Db(\sC)$ of an abelian category $\sC$ is a triangulated category with the translation functor $[1]$ induced by the shift of complexes.

If $A$ is a noetherian AS-Gorenstein algebra, then $\CM(A)$ is a Frobenius category, so $\uCM(A)$ is a triangulated category whose translation functor $[1]$ is given by 
the cosyzygy functor $\Omega^{-1}$ (see \cite[Section 4]{Bu}, \cite[Theorem 3.1]{SV}).

\subsection{Serre Duality}

For $M \in \GrMod A$ and a graded algebra automorphism $\si$ of $A$, we define the \emph{twist} $M_\si \in \GrMod A$ by
$M_\si = M$ as a graded $k$-vector space with the new right action $m*a = m\si(a)$.
Note that $\si$ induces $A_\si \cong A$ in $\GrMod A$.

If $A$ is a noetherian AS-Gorenstein algebra of dimension $d$ and Gorenstein parameter $\ell$,
then $A$ has a balanced dualizing complex $A_\nu(-\ell)[d] \in \Db(\GrMod A^{\en})$ for some graded algebra automorphism
$\nu$ of $A$. This graded algebra automorphism $\nu$ is called the \emph{(generalized) Nakayama automorphism} of $A$.
Further, the graded $A$-$A$ bimodule $\om_A := A_\nu(-\ell) \in \GrMod A^{\en}$ is called the \emph{canonical module} over $A$.
The autoequivalence $-\otimes_A \om_A: \grmod A \to \grmod A$ induces an autoequivalence
$- \Lotimes_\cA \om_{\cA} : \Db(\qgr A) \to \Db(\qgr A)$.

\begin{lem}[Serre duality \cite{dNV}] \label{lem.Se}
Let $A$ be a noetherian AS-Gorenstein algebra of dimension $d \geq 1$ and Gorenstein parameter $\ell$.
If $\gldim (\qgr A) <\infty$, then $\Db(\qgr A)$ has the Serre functor $- \Lotimes_\cA \om_{\cA}[d-1]$.
In particular, for $\cM,\cN \in \qgr A$, we have 
\[\Ext^{q}_\cA(\cM, \cN) \cong  D\Ext^{d-1-q}_\cA(\cN,\cM_\nu(-\ell))\]
where $D$ denotes the graded $k$-duality.
\end{lem}

The \emph{depth} of $M$ in $\grmod A$ is defined by $\depth M := \inf\{i \mid \Ext^i_A(k,M) \neq 0\}$.
It is well-known that if $A$ is a noetherian AS-Gorenstein algebra of dimension $d$ and $M \in \CM(A)$ is nonzero, then $\depth M =d$.

We will need to use the following lemma.

\begin{lem} \label{lem.HE}
Let $A$ be a noetherian AS-Gorenstein algebra of dimension $d \geq 2$ and Gorenstein parameter $\ell$.
If $\gldim (\qgr A) <\infty$, then for $M,N \in \CM(A)$,
\begin{enumerate}
\item $\Ext^q_\cA(\cM, \cN) \cong \Ext^q_A(M, N)$ for $0 \leq q \leq d-2$.
\item $\Ext^{d-1}_\cA(\cM, \cN) \cong D\Hom_A(N, M_\nu(-\ell))$.
\item $\Ext^{q}_\cA(\cM, \cN) =0$ for $q \geq d$.
\end{enumerate}
\end{lem}

\begin{proof} 
We only need to check when $M,N$ are nonzero.

(1) Since $\dim_k M/M_{\geq n} <\infty$ and $\depth N =d$, we have
$\lim_{n \to \infty}  \Ext^{q}_A(M/M_{\geq n}, N) =0$
for $0 \leq q \leq d-1$, so it follows from \cite[Proposition 7.2 (1)]{AZ} that 
\begin{align*}
\Ext^q_\cA(\cM, \cN) \cong \lim_{n \to \infty} \Ext^q_A(M_{\geq n}, N) \cong \Ext^q_A(M, N)
\end{align*}
for $0 \leq q \leq d-2$.

(2) Since $\depth M_{\nu}(-\ell) =\depth M=d$,
\begin{align*}
\Ext^{d-1}_\cA(\cM, \cN) \cong D\Hom_\cA(\cN, \cM_\nu(-\ell)) \cong D\Hom_A(N, M_\nu(-\ell))
\end{align*}
by \lemref{lem.Se} and (1).

(3) This follows from \lemref{lem.Se} immediately.
\end{proof}

Besides, using \lemref{lem.Se}, we can easily see that if $A$ is an AS-Gorenstein algebra of dimension $d \geq 1$
and $\gldim (\qgr A) <\infty$, then $\gldim (\qgr A)=d-1$.

For example, if $A_\e$ is as in Notation \ref{nota} with $n \geq 3$,
then $A_\e$ is an AS-Gorenstein algebra of dimension $n-1 \geq 2$, and $\Db(\qgr A_\e)$ has the Serre functor.

\subsection{Orlov's Theorem}
\begin{dfn}
Let $\sT$ be a triangulated category.
A \emph{semi-orthogonal decomposition} of  $\sT$ 
is a sequence  $\sE_0,\dots, \sE_{\ell-1}$ of strictly full triangulated subcategories such that
\begin{enumerate}
\item  for all $0 \leq i < j \leq \ell-1$ and all objects $E_i \in \sE_i, E_j \in \sE_j$, one has
$\Hom_{\sT} (E_j, E_i ) = 0$, and
\item the smallest strictly full triangulated subcategory of $\sT$ containing
$\sE_1,\dots, \sE_n$ coincides with $\sT$. 
\end{enumerate}
\end{dfn}
We use the notation $\sT = \<\sE_0,\dots,\sE_{\ell-1}\>$
for a semi-orthogonal decomposition of $\sT$ with components $\sE_0,\dots,\sE_{\ell-1}$.
Note that special and important examples of semi-orthogonal decompositions are provided by exceptional sequences
of objects (see \dfnref{dfn.res}).

Let $A$ be a noetherian AS-Gorenstein algebra.
Then the \emph{graded singularity category} of $A$ is defined by the Verdier localization 
\[ \D_{\rm Sg}^{\Z}(A) := \Db(\grmod A)/ \perf^{\Z} A \]
where $\perf^{\Z} A$ is the thick subcategory of $\Db(\grmod A)$
consisting of \emph{perfect complexes}, that is,
complexes of finite length whose terms are finitely generated projective modules.
We denote the localization functor by $\upsilon: \Db(\grmod A)\to \D_{\rm Sg}^{\Z}(A)$.
Moreover, there exists the equivalence $\uCM(A) \xrightarrow{\sim} \D_{\rm Sg}^{\Z}(A)$ by Buchweitz \cite{Bu}.

\begin{thm}[{Orlov's theorem \cite[Theorem 2.5]{O} (cf. \cite[Corollary 2.8]{IY})}] \label{thm.O}
Let $A$ be a noetherian AS-Gorenstein algebra of Gorenstein parameter $\ell$.
If $\ell > 0$, then there exists a fully faithful functor
$\Phi:=\Phi_0:\D_{\rm Sg}^{\Z}(A)\to \Db(\qgr A)$ and a semi-orthogonal decomposition
\[ \Db(\qgr A) = \< \cA(-\ell+1),\dots, \cA(-1), \cA, \Phi\D_{\rm Sg}^{\Z}(A) \> \]
where $\cA(i)$ denotes the full triangulated subcategory generated by the object $\cA(i)$.
\end{thm}

The following lemma is useful.

\begin{lem}[{\cite[Proof of Theorem 4.3]{Am}}] \label{lem.Ami}
Let $A$ be a noetherian AS-Gorenstein algebra of positive Gorenstein parameter.
If $M=M_{\geq 0}$ and $\Hom_{\grmod A}(M, A(i))=0$ for all $i\leq 0$, then $\Phi (\upsilon M)\cong  \pi M = \cM$. 
\end{lem} 

\subsection{Noncommutative Quadric Hypersurfaces}

Let $S$ be a $d$-dimensional noetherian AS-regular algebra with Hilbert series $H_S(t)= (1-t)^{-d}$.
Then $S$ is Koszul by \cite[Theorem 5.11]{S}.
Let $f \in S_2$ be a central regular element and $A=S/(f)$.
Then $A$ is a noetherian AS-Gorenstein Koszul algebra of dimension $d-1$ and Gorenstein parameter $d-2$.
There exists a central regular element $w \in A^!$ of degree $2$ such that $A^!/(w) \cong S^!$ where $A^!, S^!$ are Koszul duals of $A, S$.
Following \cite{SV}, we define 
\[C(A) := A^![w^{-1}]_0.\]
By \cite[Lemma 5.1]{SV}, $\dim_k C(A)= 2^{d-1}$.
Moreover, by \cite[Theorem 3.2]{SV}, there exists a duality
\[ \mathfrak{G}: \uCM(A) \to \Db(\mod C(A)^{\op}); \quad M \mapsto T(\RHom_A(M,k))[w^{-1}]_0\]
where $T(\RHom_A(M,k))^i_j:=\Ext^{i+j}_A(M,k)_{-j}$.
Under this duality, indecomposable non-projective graded maximal Cohen-Macaulay $A$-modules generated in degree 0 correspond to simple $C(A)$-modules
(see \cite[Lemma 4.13 (5)]{MUk}).

\section{Skew Quadric Hypersurfaces}

Throughout this section, we freely use Notation \ref{nota}. 

\subsection{Graded Maximal Cohen-Macaulay Modules}

In this subsection, we first calculate a description of $C(A_\e)$
by combinatorial methods developed in \cite{MUk}, \cite{HU}.
Next, using this, we study the form of the minimal free resolution of
an indecomposable maximal Cohen-Macaulay module over $A_\e$.
At last, we compute the extension groups between indecomposable maximal Cohen-Macaulay modules over $A_\e$.

Let $G=(V(G), E(G))$ be a finite simple graph
where $V(G)$ is the set of vertices and $E(G)$ is the set of edges. Let $v \in V(G)$ be a vertex.
We denote by $G \setminus \{v\}$ the induced subgraph of $G$ induced by $V(G) \setminus \{v\}$. Moreover, we define $N_G(v) :=\{u \in V(G) \mid uv \in E(G)\}$. 
Since reordering the vertices of the graph $G_\e$ is nothing more than reordering the variables of $A_\e$, we will frequently consider isomorphic graphs to be the same graph.

\begin{dfn}
Let $G$ be a finite simple graph.
\begin{enumerate}
\item (\cite[Definition 6.3]{MUk})
Let $v\in V(G)$. 
Then the \emph{mutation} of $G$ at $v$ is the graph $G'$ defined by $V(G')=V(G)$ and 
\[E(G')=\{ vw \mid w \in V(G) \setminus N_G(v) \} \cup E(G \setminus \{v\}).\]
\item (\cite[Definition 6.6, Lemma 6.7]{MUk})
Suppose that $G$ has an isolated vertex $i$. Let $u,v \in V(G)$ be vertices not equal to $i$.
Then the \emph{relative mutation} of $G$ at $v$ with respect to $u$ is the graph $G'$ defined by
$V(G')=V(G)$ and
\begin{align*}
E(G') =\{vw \mid w \in N_G(u) \setminus N_G(v)\} \cup \{vw \mid w \in N_G(v) \setminus N_G(u)\} \cup E(G \setminus \{v\}).
\end{align*}
\item (\cite[Lemma 6.18]{MUk})
Suppose that $G$ has two distinct isolated vertices.
We say that a graph $G'$ is a \emph{two points reduction} of $G$ if
$G'=G \setminus \{v\}$ where $v$ is an isolated vertex of $G$.
\end{enumerate}
\end{dfn}

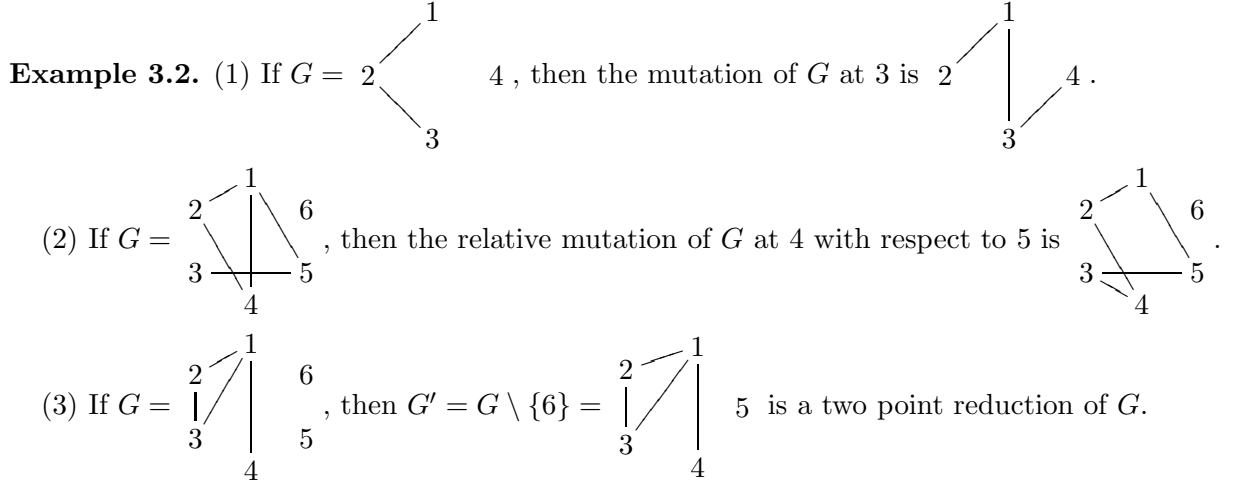
\begin{ex}
(1) If $G=  \xy /r2pc/: 
{\xypolygon4{~={90}~*{\xypolynode}~>{}}},
"1";"2"**@{-},
"2";"3"**@{-}
\endxy$, then the mutation of $G$ at $3$ is 
$\xy /r2pc/: 
{\xypolygon4{~={90}~*{\xypolynode}~>{}}},
"1";"2"**@{-},
"1";"3"**@{-},
"3";"4"**@{-}
\endxy$.

(2) If $G= \xy /r2pc/: 
{\xypolygon6{~={90}~*{\xypolynode}~>{}}},
"1";"2"**@{-},
"1";"5"**@{-},
"1";"4"**@{-},
"2";"4"**@{-},
"3";"5"**@{-},
\endxy$, then the relative mutation of $G$ at $4$ with respect to $5$ is 
$\xy /r2pc/: 
{\xypolygon6{~={90}~*{\xypolynode}~>{}}},
"1";"2"**@{-},
"1";"5"**@{-},
"2";"4"**@{-},
"3";"5"**@{-},
"3";"4"**@{-},
\endxy$.

(3) If $G = \xy /r2pc/: 
{\xypolygon6{~={90}~*{\xypolynode}~>{}}},
"1";"2"**@{-},
"1";"3"**@{-},
"2";"3"**@{-},
"1";"4"**@{-},
\endxy$, then 
$G'=G \setminus \{6 \} = \xy /r2pc/: 
{\xypolygon5{~={72}~*{\xypolynode}~>{}}},
"1";"2"**@{-},
"1";"3"**@{-},
"2";"3"**@{-},
"1";"4"**@{-},
\endxy$
is a two point reduction of $G$.
\end{ex}

\begin{lem}\label{lem.mrm}
 Let $r=\mnull_{\F_2} \De_{\e}$. Then $G_\e$ can be transformed into
\[ \xymatrix@R=0.4pc @C=1pc{
 &1\ar@{-}[dd] &3\ar@{-}[dd] &&n-r-2\ar@{-}[dd]\\
 &  &  &\cdots &&n-r &n-r+1 &\cdots &n \\
 &2 &4 &&n-r-1
}
\]
by applying mutation and relative mutation several times up to isomorphism.
\end{lem}

\begin{proof}
By combining \cite[Lemma 3.1]{HU} and \cite[Lemma 3.2]{HU}, we see that $G_\e$ can be transformed into
\[ \xymatrix@R=0.4pc @C=1pc{
 &1\ar@{-}[dd] &3\ar@{-}[dd] &&n-s-2\ar@{-}[dd]\\
 &  &  &\cdots &&n-s &n-s+1 &\cdots &n \\
 &2 &4 &&n-s-1
}
\]
for some $s\geq 0$, by applying mutation and relative mutation several times up to isomorphism.
Then it follows from \cite[Lemma 3.3]{HU} that
$\mnull_{\F_2} \De_{\e}= n+1-\rank_{\F_2} \De_{\e}= n+1-(n-s+1)=s$, so we get $r=s$.
\end{proof}

\begin{lem}\label{lem.cg}
Let $G_{\e}, G_{\e'}$ be the graphs associated with $\e, \e'$. 
\begin{enumerate}
\item {\rm (\cite[Lemma 6.5]{MUk})} If $G_{\e'}$ is obtained from $G_{\e}$ by a mutation, then $C(A_{\e})\cong C(A_{\e'})$. 
\item {\rm (\cite[Lemma 6.7]{MUk})} If $G_{\e'}$ is obtained from $G_{\e}$ by a relative mutation, then $C(A_{\e})\cong C(A_{\e'})$.
\item {\rm (\cite[Lemma 6.18]{MUk})} If $G_{\e'}$ is obtained from $G_{\e}$ by a two points reduction, then $C(A_{\e})\cong C(A_{\e'})^{2}$. 
\end{enumerate}
\end{lem}

\begin{lem}\label{lem.ck}
For $m \in \N$, if
\[ \xymatrix@R=0.4pc{
 &1\ar@{-}[dd] &3\ar@{-}[dd] &&2m-1\ar@{-}[dd]\\
G_\e= &  &  &\cdots &&2m+1, \\
 &2 &4 &&2m
}
\]
then $C(A_\e) \cong M_{2^m}(k)$.
\end{lem}

\begin{proof}
Since $\ep_{2m+1,i}=\ep_{i,2m+1}=-1$ for any $1 \leq i\leq 2m$, it follows from \cite[lemma 3.1(3)]{Uk} that 
\begin{align*}
C(A_\e) &\cong k\ang{t_1,\dots,t_{2m}}/(t_it_j+\e_{2m+1,i}\e_{ij}\e_{j,2m+1}t_jt_i, t_i^2-1)_{1 \leq i,j\leq 2m, i\neq j}\\
&\cong k\ang{t_1,\dots,t_{2m}}/(t_it_j+\e_{ij}t_jt_i, t_i^2-1)_{1 \leq i,j\leq 2m, i\neq j},
\end{align*}
so it is enough to show
\begin{align} \label{eq.c}
\Ga_{m} := k\ang{t_1,\dots,t_{2m}}/(t_it_j+\e_{ij}t_jt_i, t_i^2-1)_{1 \leq i,j\leq 2m, i\neq j} \cong M_{2^m}(k).
\end{align}
We use induction on $m$. The case $m=0$ is clear. Suppose that $m \geq 1$ and $\La_{m-1} \cong M_{2^{m-1}}(k)$.
Since $\ep_{2m-1,i}=\ep_{2m,i}=-1$ for any $1 \leq i \leq 2m-2$, 
we see that $t_{2m-1}$ and $t_{2m}$ commute with $t_1,\dots,t_{2m-2}$ in $\Ga_{m}$, so we have
\begin{align*}
\Ga_{m} &\cong k\ang{t_1,\dots,t_{2m-2}}/(t_it_j+\e_{ij}t_jt_i, t_i^2-1)_{1 \leq i,j\leq 2m-2, i\neq j} \\
&\hspace{10mm} \otimes_k k\ang{t_{2m-1},t_{2m}}/(t_{2m-1}t_{2m}+\e_{2m-1,2m}t_{2m}t_{2m-1}, t_{2m-1}^2-1, t_{2m}^2-1)\\
&\cong \Ga_{m-1} \otimes_k k\ang{t_{2m-1},t_{2m}}/(t_{2m-1}t_{2m}+t_{2m}t_{2m-1}, t_{2m-1}^2-1, t_{2m}^2-1)\\
&\cong M_{2^{m-1}}(k) \otimes_k M_{2}(k)\\
&\cong M_{2^{m}}(k)
\end{align*}
as desired. 
\end{proof}

\begin{lem} \label{lem.c} 
Let $r = \mnull_{\F_2} \De_{\e}, \al =2^r, \be = 2^{\frac{n-r-1}{2}}$.
Then $C(A_\e) \cong M_{\be}(k)^{\al}$.
\end{lem}

\begin{proof}
By \lemref{lem.mrm}, $G_\e$ can be transformed into
\[ \xymatrix@R=0.4pc @C=1pc{
 &1\ar@{-}[dd] &3\ar@{-}[dd] &&n-r-2\ar@{-}[dd]\\
G_{\e'} := &  &  &\cdots &&n-r &n-r+1 &\cdots &n \\
 &2 &4 &&n-r-1
}
\]
by applying mutation and relative mutation several times up to isomorphism. Notice that $G_{\e'}$ consists of $\frac{n-r-1}{2}$ isolated edges and $r+1$ isolated vertices.
Using two points reductions $r$ times, $G_{\e'}$ becomes
\[ \xymatrix@R=0.4pc{
 &1\ar@{-}[dd] &3\ar@{-}[dd] &&n-r-2\ar@{-}[dd]\\
G_{\e''} := &  &  &\cdots &&n-r. \\
 &2 &4 &&n-r-1
}
\]
By \lemref{lem.cg}, $C(A_\e) \cong C(A_{\e'}) \cong C(A_{\e''})^{\al}$.
Moreover, by \lemref{lem.ck}, $C(A_{\e''}) \cong M_\be (k)$.
Hence we have $C(A_\e) \cong M_\be (k)^{\al}$.
\end{proof}

We define
\[\M_\e := \{ M \in \CM(A_\e) \mid \text{$M$ is indecomposable, non-projective, and generated in degree 0} \}/{\cong}\]

\begin{lem} \label{lem.x}
Let $r = \mnull_{\F_2} \De_{\e}, \al =2^r, \be = 2^{\frac{n-r-1}{2}}$.
\begin{enumerate}
\item
$\M_\e$ consists of $\al$ modules, say $X_1, X_2, \dots, X_{\al}$.
\item
Every indecomposable non-projective graded maximal Cohen-Macaulay $A_\e$-module is isomorphic to $X_i(j)$
for some $1\leq i \leq \al$ and some $j \in \Z$.
\item If $X \in \M_\e$, then there exists an exact sequence 
\begin{align}\label{eq.es}
0 \to X(-2) \to A_\e(-1)^{\be} \to A_{\e}^{\be} \to X \to 0
\end{align}
in $\grmod A_{\e}$ for every $1 \leq i\leq \al$.
\item If $X \in \M_\e$, then $\Om^iX(i) \in \M_\e$ for any $i \in \Z$.
\end{enumerate}
\end{lem}

\begin{proof}
(1) By \cite[Lemma 4.13 (5)]{MUk}, the duality $\mathfrak{G}$ (which appeared in Section 2.4) induces a bijection between $\M_\e$ and isomorphism classes of simple $C(A_\e)$-modules.
Since $C(A_\e) \cong M_\be (k)^{\al}$ by \lemref{lem.c}, $C(A_\e)$ has $\al$ simple modules up to isomorphism, so the claim holds.

(2) Since $C(A_\e) \cong M_\be (k)^{\al}$ is semisimple, the result follows from \cite[Lemma 5.2]{MUk}.

(3) Since $C(A_\e) \cong M_\be (k)^{\al}$, every simple $C(A_\e)$-module has dimension $\be$,
so it follows from the proof of (1) that $\dim_k \mathfrak{G}(X)= \be$.
Since $X$ has a linear free resolution by \cite[Proposition 7.8 (1)]{MUm}, we see
\[ \mathfrak{G}(X) = \Ext^*_{A_\e}(X, k)[w^{-1}]_0 \]
where $\Ext^*_{A_\e}(X, k)$ is defined by $\bigoplus_{j \in \N}\Ext^j_{A_\e}(X, k)$.
Thus we get 
\[
\beta = \dim_k \Ext^*_{A_\e}(X, k)[w^{-1}]_0=\dim_k \Ext^{2i}_{A_\e}(X, k)
\] for $i \gg 0$,
so the $2i$-th term of the minimal free resolution of $X$ is $A_\e(-2i)^\be$ for $i \gg 0$.
By \cite[Proposition 6.2]{MUm}, this resolution is obtained from a noncommutative matrix factorization of rank $\be$,
so the $i$-th term must be $A_\e(-i)^\be$ for every $i \geq 0$.
Since $f_\e$ is a central element of $S_\e$, we have $\Omega^{2}X \cong X(-2)$ by \cite[Lemma 4.11]{MUm}, and hence the result.

(4) This follows from (3).
\end{proof}

For a graded vector space $V=\bigoplus_{i \in \Z} V_i$ such that $\dim_k V_i < \infty$ for all $i \in \Z$, we define the \emph{Hilbert series} of $V$ by
$H_V(t) :=\sum_{i \in \Z} (\dim_kV_i)t^i \in \Z[[t, t^{-1}]]$.

\begin{lem}\label{lem.hil}
Let $r = \mnull_{\F_2} \De_{\e}, \al =2^r, \be = 2^{\frac{n-r-1}{2}}$. 
Then the following hold.
\begin{enumerate}
\item $H_{A_\e}(t)=(1+t)(1-t)^{-(n-1)}$.
\item If $X \in \M_\e$, then $H_{X}(t)=\be (1-t)^{-(n-1)}$.
\item If $X \in \M_\e$, then $H_{\Hom_{A_\e}(X,A_\e)}(t)=\be t(1-t)^{-(n-1)}$.
\end{enumerate}
\end{lem}

\begin{proof}
(1) The exact sequence
$0 \to S_\e(-2) \to S_\e \to A_\e \to 0$ implies 
$t^2H_{S_\e}(t)-H_{S_\e}(t)+H_{A_\e}(t)=0$.
Since $H_{S_\e}(t) = (1-t)^{-n}$, we have
$H_{A_\e}(t) = (1-t^2)H_{S_\e}(t)= (1+t)(1-t)^{-(n-1)}$.

(2) By (\ref{eq.es}), we have $H_X(t)=\be (1+t)^{-1}H_{A_\e}(t)$, so the result follows from (1).

(3) Since $X \in \CM(A_\e)$, applying $\Hom_{A_\e}(-,A_\e)$ to (\ref{eq.es}) yields an exact sequence
\begin{align*}
0 \to \Hom_{A_\e}(X,A_\e) \to A_\e^{\be} \to A_{\e}(1)^{\be} \to \Hom_{A_\e}(X,A_\e)(2) \to 0,
\end{align*}
so we have $(t^{-2}-1)H_{\Hom_{A_\e}(X,A_\e)}(t) = \be (t^{-1}-1)H_{A_\e}(t)$. Thus the result follows from (1).
\end{proof}

\begin{lem}\label{lem.extx}
For $X, Y \in \M_\e$, $q\geq 1$ and $i \in \Z$, 
\[\Ext^q_{\grmod A_\e}(X,Y(i)) \cong \begin{cases} k &\text{if}\; i=-q \; \text{and} \; X \cong \Om^{i}Y(i), \\
0 &\text{otherwise.} \end{cases}\]
\end{lem}

\begin{proof}
Since $X, Y \in \M_\e$,
\begin{align*}
\Ext^q_{\grmod A_\e}(X,Y(i))
&\cong \Hom_{\uCM(A_\e)}(X,Y(i)[q])\\
&\cong \Hom_{\uCM(A_\e)}(X,\Om^{-i}\Om^iY(i)[q])\\
&\cong \Hom_{\uCM(A_\e)}(X,\Om^iY(i)[i+q])\\
&\cong \Hom_{\Db(\mod C(A_\e)^{\op})}(\mathfrak{G}(\Om^iY(i)[i+q]),\mathfrak{G}(X))\\
&\cong \Hom_{\Db(\mod C(A_\e)^{\op})}(\mathfrak{G}(\Om^iY(i))[-i-q],\mathfrak{G}(X))
\quad (\because \text{$\mathfrak{G}$ is a duality})\\
&\cong \Hom_{\Db(\mod C(A_\e)^{\op})}(\mathfrak{G}(\Om^iY(i)),\mathfrak{G}(X)[i+q]).
\end{align*}
By \lemref{lem.x} (4), $X, \Om^iY(i) \in \M_\e$, so $\mathfrak{G}(X), \mathfrak{G}(\Om^iY(i))$ are simple $C(A)$-modules.
Thus we have
\begin{align}\label{eq.gext}
\Hom_{\Db(\mod C(A_\e)^{\op})}(\mathfrak{G}(\Om^iY(i)),\mathfrak{G}(X)[i+q]) \cong \Ext^{i+q}_{C(A_\e)^{\op}}(\mathfrak{G}(\Om^iY(i)),\mathfrak{G}(X)).
\end{align}
Since $C(A_\e) \cong M_\be (k)^{\al}$ is semisimple, (\ref{eq.gext}) is isomorphic to $k$ if $i+q = 0$ and $X \cong \Om^iY(i)$,
and it is zero otherwise.
\end{proof}

\begin{lem}\label{lem.homx}
For $X, Y \in \M_\e$ and $i \leq 0$,
\[\Hom_{\grmod A_\e}(X,Y(i)) \cong \begin{cases} k &\text{if}\; i=0 \; \text{and} \; X \cong Y, \\
0 &\text{otherwise.} \end{cases}\]
\end{lem}

\begin{proof}
If $i<0$, then the result follows from the fact that $X$ is generated in degree 0.
Therefore, let us check the case $i=0$.
Applying $\Hom_{A_\e}(-,Y)$ to the exact sequence $0 \to X(-2) \to A_\e(-1)^{\be} \to \Om X \to 0$, we have an exact sequence
\begin{align}\label{eq.homx}
0 \to \Hom_{A_\e}(\Om X,Y) \to \Hom_{A_\e}(A_\e(-1)^{\be},Y) \to \Hom_{A_\e}(X(-2),Y) \to \Ext^1_{A_\e}(\Om X,Y) \to 0.
\end{align}
Since $\Hom_{A_\e}(A_\e(-1)^{\be},Y)_{-2} \cong Y^\be_{-1}=0$,
taking degree $-2$ parts of (\ref{eq.homx}) yields
\begin{align}\label{eq.ghom}
\Hom_{\grmod A_\e}(X,Y) &\cong \Hom_{A_\e}(X(-2),Y)_{-2} \cong \Ext^1_{A_\e}(\Om X,Y)_{-2} \cong \Ext^2_{A_\e}(X,Y)_{-2} \notag \\ 
&\cong \Ext^2_{\grmod A_\e}(X,Y(-2)).
\end{align}
By \lemref{lem.extx}, (\ref{eq.ghom}) is isomorphic to $k$ if $X \cong \Om^{-2}Y(-2) \cong Y$, and it is zero otherwise.
\end{proof}

\begin{lem} \label{lem.Amx}
Let $r=\mnull_{\F_2} \De_{\e}$ and $\al=2^r$.
Then we have a semi-orthogonal decomposition $\Phi\D_{\rm Sg}^{\Z}(A_\e) = \langle \cX_1, \cX_2, \dots, \cX_\al \rangle$.
\end{lem} 

\begin{proof}
Since $C(A_\e) \cong M_\be (k)^{\al}$ is semisimple and $\mathfrak{G}(X_\al), \dots, \mathfrak{G}(X_1)$ form a complete set of simple modules, we have  a semi-orthogonal decomposition $\Db(\mod C(A_\e)^{\op}) = \< \mathfrak{G}(X_\al), \dots, \mathfrak{G}(X_1)\>$.
Thus we see $\uCM(A_\e) = \< X_1, \dots, X_\al \>$.
Under the equivalence $\uCM(A_\e) \xrightarrow{\sim} \D_{\rm Sg}^{\Z}(A_\e)$, $X_i$ corresponds to $\upsilon X_i$,
so $\D_{\rm Sg}^{\Z}(A_\e) = \< \upsilon X_1, \dots, \upsilon X_\al \>$.
By Lemma \ref{lem.hil} and Lemma \ref{lem.Ami}, it follows that $\Phi (\upsilon X_i)\cong \cX_i$.
Hence we get  $\Phi\D_{\rm Sg}^{\Z}(A_\e) = \< \Phi(\upsilon X_1), \dots, \Phi(\upsilon X_\al) \> = \langle \cX_1,\dots, \cX_\al \rangle$.
\end{proof}

\subsection{Exceptional Sequences}
In this subsection,  we show that $\Db(\qgr A_\e)$ has a full exceptional sequence. 

\begin{dfn} \label{dfn.res}
Let $\sT$ be a $k$-linear triangulated category.  
\begin{enumerate}
\item An object $E$  of $\sT$ is \emph{exceptional} if
$\End_{\sT}(E)=k$
and 
$\Hom_{\sT}(E, E[q])=0$ for every $q\neq 0$.
\item A sequence of objects $(E_0, \dots, E_{\ell-1})$ in $\sT$ is an \emph{exceptional sequence} if
\begin{enumerate}
\item
$E_i$ is an exceptional object for every $i=0, \dots, \ell-1$, and 
\item 
$\Hom_{\sT}(E_j, E_i[q])=0$ for every $q$ and every $0\leq i<j\leq \ell-1$.
\end{enumerate} 
\item An exceptional sequence $(E_0, \dots, E_{\ell-1})$ in $\sT$ is \emph{full} if
$\sT$ is generated by $E_0, \dots, E_{\ell-1}$, that is, 
the smallest strictly full triangulated subcategory of $\sT$ containing $E_0, \dots, E_{\ell-1}$ is equal to $\sT$. 
\item An exceptional sequence $(E_0, \dots, E_{\ell-1})$ in $\sT$ is \emph{strong} if $\Hom_{\sT}(U, U[q])=0$ for all $q\neq 0$, where $U= E_0 \oplus \cdots \oplus E_{\ell-1}$.
\end{enumerate}
\end{dfn} 
If $(E_0, \dots, E_{\ell-1})$ is a full exceptional sequence in $\sT$, then we get a semi-orthogonal
decomposition $\sT=\<E_0, \dots, E_{\ell-1}\>$.

\begin{lem} \label{lem.2a}
Let $r=\mnull_{\F_2} \De_{\e}$ and $\al=2^r$.
Assume $n \geq 3$.
Then 
\begin{align}\label{eq.fes}
(\cA_\e(-n+3), \dots, \cA_\e(-1), \cA_\e, \cX_1, \cX_2, \dots, \cX_\al)
\end{align}
is a full exceptional sequence in $\Db(\qgr A_\e)$.
\end{lem}

\begin{proof}
For simplicity, we write $A:=A_\e$.
For $-n+3 \leq s \leq 0$ and $1 \leq i \leq n-3$,
\begin{align*}
&\End_{\qgr A}(\cA(s))\cong \End_{\grmod A}(A(s))\cong A_0\cong k \quad\text{by Lemma \ref{lem.HE} (1)},\\
&\Ext^q_{\qgr A}(\cA(s),\cA(s)) \cong \Ext^q_{\grmod A}(A(s),A(s))=0 \quad (1\leq q \leq n-3) \quad\text{by Lemma \ref{lem.HE} (1)},\\
&\Ext^{n-2}_{\qgr A}(\cA(s),\cA(s)) \cong D\Hom_{\grmod A}(A(s),A(s-n+2)) \cong D(A_{-n+2}) = 0 \quad\text{by Lemma \ref{lem.HE} (2)},\\
&\Ext^q_{\qgr A}(\cA(s),\cA(s)) = 0 \quad (q \leq 0 \; \text{or} \; n-1 \leq q) \quad\text{by Lemma \ref{lem.HE} (3)},\\
&\End_{\qgr A}(\cX_i)\cong \End_{\grmod A}(X_i) \cong k \quad\text{by Lemmas \ref{lem.HE} (1), \ref{lem.homx}},\\
&\Ext^q_{\qgr A}(\cX_i,\cX_i) \cong \Ext^q_{\grmod A}(X_i,X_i)=0 \quad (1\leq q \leq n-3) \quad\text{by Lemmas \ref{lem.HE} (1), \ref{lem.extx}},\\
&\Ext^{n-2}_{\qgr A}(\cX_i,\cX_i) \cong D\Hom_{\grmod A}(X_i, (X_i)_\nu(-n+2)) = 0 \quad\text{by Lemmas \ref{lem.HE} (2), \ref{lem.homx}},\\
&\Ext^q_{\qgr A}(\cX_i,\cX_i) = 0 \quad (q \leq 0 \; \text{or} \; n-1 \leq q) \quad\text{by Lemma \ref{lem.HE} (3)},
\end{align*}
so objects in  (\ref{eq.fes}) are exceptional.
Moreover, since we have semi-orthogonal decompositions
\begin{align*}
\Db(\qgr A) = \< \cA(-\ell+1),\dots, \cA(-1), \cA, \Phi\D_{\rm Sg}^{\Z}(A) \> \;\;\text{and}\;\;
\Phi\D_{\rm Sg}^{\Z}(A) =  \< \cX_1, \cX_2, \dots, \cX_\al \>
\end{align*}
by \thmref{thm.O} and \lemref{lem.Amx},  it follows that  (\ref{eq.fes}) is a full exceptional sequence.
\end{proof}

\begin{lem} \label{lem.2b}
Let $r=\mnull_{\F_2} \De_{\e}$ and $\al=2^r$.
Assume $n \geq 3$.
If $\cU= \cA_\e(-n+3) \oplus \cdots \oplus \cA_\e \oplus \cX_1 \oplus \cdots \oplus \cX_\al$, then
\begin{align*}
\Ext^q_{\qgr A_\e}(\cU, \cU \otimes_{\cA_\e}(\omega _{\cA_\e}^{-1})^{\otimes i})=0
\end{align*}
for $q \neq 0$ and $i \geq 0$.
\end{lem}

\begin{proof}
For simplicity, we write $A:=A_\e$.
Since $-\otimes_{\cA}\omega _{\cA}^{-1} \cong (-)_{\nu^{-1}}(n-2)$,
we have $\cU \otimes_{\cA}(\omega _{\cA}^{-1})^{\otimes i} \cong \cU_{\nu^{-i}}(i(n-2)) \cong \cU(i(n-2))$ where $\cU_{\nu^{-i}}= \pi(U_{\nu^{-i}})$,
so it is enough to show that
\begin{align*}
&\Ext^q_{\qgr A}(\cA(s),\cA(t)) = 0 \quad (-n+3 \leq s \leq 0,\; -n+3 \leq t),\\
&\Ext^q_{\qgr A}(\cA(s),\cX_j(t')) = 0 \quad (-n+3 \leq s \leq 0,\; 1 \leq j \leq \al,\; 0 \leq t'),\\
&\Ext^q_{\qgr A}(\cX_j,\cA(t)) = 0 \quad (1 \leq j \leq \al,\; -n+3 \leq t),\\
&\Ext^q_{\qgr A}(\cX_j,\cX_{l}(t')) = 0 \quad (1 \leq j, l \leq \al,\; 0 \leq t').
\end{align*}
If $q \geq n-1$, then these are true by \lemref{lem.HE} (3). Furthermore, one can calculate
\begin{align*}
\Ext^q_{\qgr A}(\cA(s),\cA(t))&\cong \Ext^q_{\grmod A}(A(s),A(t))=0 \quad (1\leq q \leq n-3) \quad\text{by Lemma \ref{lem.HE} (1)},\\
\Ext^{n-2}_{\qgr A}(\cA(s),\cA(t))&\cong D\Hom_{\grmod A}(A(t),A(s-n+2))\\
&\cong D(A_{s-t-n+2})=0\quad\text{by Lemma \ref{lem.HE} (2)},\\
\Ext^q_{\qgr A}(\cA(s),\cX_j(t'))&\cong \Ext^q_{\grmod A}(A(s),X_j(t'))=0 \quad (1\leq q \leq n-3)\quad\text{by Lemma \ref{lem.HE} (1)},\\
\Ext^{n-2}_{\qgr A}(\cA(s),\cX_j(t'))&\cong D\Hom_{\grmod A}(X_j(t'),A(s-n+2))\\
&\cong D(\Hom_A(X_j,A)_{s-t'-n+2})=0\quad\text{by Lemmas \ref{lem.HE} (2), \ref{lem.hil} (3)},\\
\Ext^q_{\qgr A}(\cX_j,\cA(t))&\cong \Ext^q_{\grmod A}(X_j, A(t))=0 \quad (1\leq q \leq n-3)\quad\text{by Lemma \ref{lem.HE} (1)},\\
\Ext^{n-2}_{\qgr A}(\cX_j,\cA(t))&\cong D\Hom_{\grmod A}(A(t),(X_j)_\nu(-n+2))\\
&\cong D(((X_j)_\nu)_{-t-n+2})=0\quad\text{by Lemmas \ref{lem.HE} (2), \ref{lem.hil} (2)},\\
\Ext^q_{\qgr A}(\cX_j,\cX_{l}(t'))&\cong  \Ext^q_{\grmod A}(X_j,X_{l}(t'))=0 \quad (1\leq q \leq n-3)\quad\text{by Lemmas \ref{lem.HE} (1), \ref{lem.extx}},\\
\Ext^{n-2}_{\qgr A}(\cX_j,\cX_{l}(t'))&\cong D\Hom_{\grmod A}(X_{l}(t'),(X_j)_\nu(-n+2))\\
&\cong D\Hom_{\grmod A}(X_{l},(X_j)_\nu(-t'-n+2))=0\quad\text{by Lemmas \ref{lem.HE} (2), \ref{lem.homx}}.
\end{align*}
Thus the desired result holds. 
\end{proof}

\subsection{Ampleness}
In this subsection, we show that the full exceptional sequence (\ref{eq.fes}) induces
an ample pair in the sense of Artin and Zhang \cite{AZ}. 

Let $\sC$ be a $k$-linear abelian category.
We call $(L, \si)$ an \emph{algebraic pair} for $\sC$ if it consists of
an object $L \in \sC$ and a $k$-linear autoequivalence $\si \in \Aut_k \sC$.

\begin{dfn}    
An algebraic pair $(L, \si)$ for a $k$-linear abelian category $\sC$ is \emph{ample} if 
\begin{enumerate}
\item[(A1)] for every object $M \in \sC$, there exists an epimorphism $\bigoplus_{j=1}^p \si^{-i_j}L \to M$ in $\sC$ for some $i_1, \dots, i_p \in \N$, and  
\item[(A2)] for every epimorphism $\phi: M \to N$ in $\sC$, there exists $m\in \Z$ such that 
$$\Hom_{\sC}(\si^{-i}L, \phi):\Hom_{\sC}(\si^{-i}L, M)\to \Hom_{\sC}(\si^{-i}L, N)$$
is surjective for every $i\geq m$.  
\end{enumerate}
\end{dfn}

\begin{lem} \label{lem.2c}
Let $r=\mnull_{\F_2} \De_{\e}$ and $\al=2^r$.
Assume $n \geq 3$.
If $\cU= \cA_\e(-n+3) \oplus \cdots \oplus \cA_\e \oplus \cX_1 \oplus \cdots \oplus \cX_\al$, then
$(\cU, -\otimes_{\cA_\e}\omega _{\cA_\e}^{-1})$ is ample for $\qgr A$.
\end{lem}

\begin{proof} 
For simplicity, we write $A:=A_\e$.
First, note that
$\cU \otimes_{\cA} (\omega _{\cA}^{-1})^{\otimes -i} \cong \cU_{\nu^i}(i(-n+2)) \cong \cU(i(-n+2))$
for any $i \in \Z$.
Since every $\cA(-j)$ with $j \in \N$ is contained in
$\cU \otimes_{\cA} (\omega _{\cA}^{-1})^{\otimes -i} \cong \cU(i(-n+2))$ as a direct summand for some $i \in \N$,
we see that the condition (A1) is satisfied. 

Let $\phi: \cM \to \cN$ be an epimorphism in $\qgr A$. It gives a short exact sequence
\begin{align} \label{eq.a.ses}
0 \to \cK \to \cM \to \cN \to 0
\end{align}
in $\qgr A$. Then we take $K \in \grmod A$ such that $\pi K = \cK$.
By \cite[Lemma 5.7]{Ucm}, $\dim_k \Ext^1_A(X_j, K) < \infty$ for all $1\leq j \leq \al$,
so it follows from \cite[Corollary 7.3 (2)]{AZ} that $\Ext^1_{\cA}(\cX_j, \cK)$ and $\Ext^1_{\cA}(\cA, \cK)$ are right bounded for all $1\leq j \leq \al$.
Thus there exists $m \in \N$ such that 
\begin{align*}
\Ext^1_{\qgr A}(\cX_j(-i), \cK) = \Ext^1_{\cA}(\cX_j, \cK)_{i}=0 \;\; \text{and}\;\; \Ext^1_{\qgr A}(\cA(-i), \cK)= \Ext^1_{\cA}(\cA, \cK)_{i}=0
\end{align*}
for all $1\leq j \leq \al$ and all $i \geq m$. It follows that
$\Ext^1_{\qgr A}(\cU(-i), \cK)=0$
for all $i \geq m$.

For any $i \geq m$, applying $\Hom_{\cA}(\cU(i(-n+2)),-)$ to (\ref{eq.a.ses}), we have an exact sequence
\[ \xymatrix@C=1pc@R=1pc{
\Hom_{\cA}(\cU(i(-n+2)),\cM) \ar[r]^{\psi} &\Hom_{\cA}(\cU(i(-n+2)),\cN) \ar[r] &\Ext^1_{\cA}(\cU(i(-n+2)),\cK).
}\]
Since $-i(-n+2)=i(n-2) \geq i\geq m$, we have $\Ext^1_{\cA}(\cU(i(-n+2)),\cK)=0$, so the map $\psi = \Hom_{\cA}(\cU(i(-n+2)), \phi)$ is surjective. Hence the condition (A2) is satisfied.
\end{proof}

\subsection{Main Result}
In this subsection, we give the proof of our main result (\thmref{thm.main}).
In the context of noncommutative algebraic geometry, Minamoto \cite{Min} introduced a nice class of finite-dimensional algebras of finite global dimension.

\begin{dfn}
Let $R$ be a finite-dimensional algebra and $L$ a two-sided tilting complex of $R$.
\begin{enumerate}
\item{} We say that $L$ is \emph{extremely ample} if 
\begin{enumerate}
\item
$h^q(L^{\Lotimes i})=0$ for all $q\neq 0$ and all $i\geq 0$, and 
\item
$(\D^{L, \geq 0}, \D^{L, \leq 0})$ is a t-structure on $\Db(\mod R)$ where
\begin{align*}
\D^{L, \geq 0} & :=\{M\in \Db(\mod R)\mid h^q(M\Lotimes _R L^{\Lotimes i})=0 \; \textnormal { for all } q<0, i\gg 0\}, \\
\D^{L, \leq 0} & :=\{M\in \Db(\mod R)\mid h^q(M\Lotimes _R L^{\Lotimes i})=0 \; \textnormal { for all } q>0, i\gg 0\}.
\end{align*}
\end{enumerate}
\item Assume that $\gldim R=m<\infty$. Then the \emph{canonical module} of $R$ is defined as the two-sided tilting complex
$\om_R := DR[-m]$.  
\item We say that $R$ is \emph{extremely Fano} if $\om _R^{-1}:=\RHom_R(\om _R, R)$ is extremely ample.
\end{enumerate}
\end{dfn}

For an extremely Fano algebra $R$ of $\gldim R =m$, the \emph{$(m+1)$-preprojective algebra} of $R$
is defined as the tensor algebra 
\[ \Pi R:=T_R(\Ext_R^m(DR, R))\]
of the $R$-$R$ bimodule $\Ext_R^m(DR, R)$.

For example, the path algebra $kQ$ of a finite acyclic quiver $Q$ of infinite-representation type is an extremely Fano algebra of global dimension $1$.
In this case, $\Pi kQ$ is isomorphic to the usual preprojective algebra of $kQ$ (see \cite{Min}).

We here recall the definition of a Calabi-Yau algebra.

\begin{dfn} 
A locally finite graded algebra $A$ is called \emph{(bimodule) Calabi-Yau of dimension $d$ and Gorenstein parameter $\ell$} if
$A \in \perf^{\Z}A^{\en}$ and $\RHom_{A^{\en}}(A, A^{\en}) \cong {A}(\ell)[-d]$ in $\D(\GrMod A^{\en})$.
\end{dfn}

The following result provides a strong connection between
noncommutative algebraic geometry and
representation theory of finite-dimensional algebras.

\begin{thm}[{\cite[Theorem 4.2, Theorem 4.12 (1)]{MM}, \cite[Theorem 4.36]{HIO}}]
If a finite-dimensional algebra $R$ is extremely Fano of global dimension $d$, then $\Pi R$ is a coherent Calabi-Yau algebra of dimension $d+1$ and Gorenstein parameter $1$.
Conversely, if a graded algebra $A$ is coherent Calabi-Yau of dimension $d+1$ and Gorenstein parameter $1$, then $A_0$ is an extremely Fano algebra of dimension $d$.
\end{thm}

Moreover, we have the following result, which is an immediate application of the results in \cite{MUq}.

\begin{thm} \label{thm.MUq}
Let $A$ be a noetherian AS-Gorenstein algebra of dimension $d\geq 1$.
Suppose that
\begin{enumerate}
\item[(\Rnum{1})] $\gldim(\qgr A)<\infty$, and
\item[(\Rnum{2})] a sequence of objects $(\cE_0,\dots,\cE_{\ell-1})$ in $\qgr A$ satisfies
\begin{enumerate}
\item $(\cE_0,\dots,\cE_{\ell-1})$ is a full exceptional sequence in $\Db(\qgr A)$,
\item $\Ext^q_{\cA}(\cU, \cU \otimes_{\cA}(\omega _{\cA}^{-1})^{\otimes i})=0$ for all $q\neq 0$ and all $i\geq 0$, and
\item $(\cU, -\otimes_{\cA}\omega_{\cA}^{-1})$ is ample for $\qgr A$,
\end{enumerate}
where $\cU:= \cE_0 \oplus \cdots \oplus \cE_{\ell-1}$.
\end{enumerate}
Then the following assertions hold.
\begin{enumerate}
\item $R:=\End_{\cA}(\cU)$ is an extremely Fano algebra of global dimension $d-1$.
\item There exists an equivalence of triangulated categories $\Db(\qgr A) \cong \Db(\mod R)$.
\item $\Pi R$ is a right noetherian graded Calabi-Yau algebra of dimension $d$
(in particular, an AS-regular algebra over $R$ in the sense of Minamoto and Mori \cite{MM}).
\item There exists an equivalence $\qgr A \cong \qgr \Pi R$.
\end{enumerate}
\end{thm}

\begin{proof}
First, note that $\gldim (\qgr A)=d-1$.
By (\Rnum{2}a) and (\Rnum{2}b), one can verify that $\cU$ is a regular tilting object of $\Db(\qgr A)$ in the sense of \cite[Definition 3.9]{MUq}.
Hence the assertions follow from \cite[Theorem 3.11, Theorem 4.1, Corollary 4.3]{MUq}.
\end{proof}

We are now ready to turn to the main result of this paper.

\begin{thm} \label{thm.main}
Let $r = \mnull_{\F_2} \De_{\e}, \al =2^r, \be = 2^{\frac{n-r-1}{2}}$.
Assume that $n\geq 3$.
Then $\Db(\qgr A_\e)$ has a full strong exceptional sequence of the form
\[
\Db(\qgr A_\e)=
\< \cA_\e(-n+3), \dots, \cA_\e(-1), \cA_\e, \cX_1, \cX_2, \dots, \cX_\al \>.
\]
Moreover, if $\cU= \cA_\e(-n+3) \oplus \cdots \oplus \cA_\e \oplus \cX_1 \oplus \cdots \oplus \cX_\al$, then 
the following statements hold.
\begin{enumerate}
\item $\La_\e := \End_{\qgr A_\e}{\cU}$ is an extremely Fano algebra of $\gldim \La_\e = n-2$.
\item There exists an equivalence of triangulated categories $\Db(\qgr A_\e) \cong \Db(\mod \La_\e)$.
\item $\Pi \La_\e$ is a right noetherian graded Calabi-Yau algebra of dimension $n-1$.
\item There exists an equivalence $\qgr A_\e \cong \qgr \Pi\La_\e$.
\end{enumerate}
\end{thm}

\begin{proof}
Note that $A_\e$ is AS-Gorenstein of dimension $n-1$.
By \cite[Theorem 1.4]{HU},  $\gldim (\qgr A_\e)<\infty$.
By \lemref{lem.2a}, we see that
$(\cA_\e(-n+3), \dots, \cA_\e, \cX_1, \dots, \cX_\al)$ is a full exceptional sequence in $\Db(\qgr A_\e)$.
Moreover, by \lemref{lem.2b},
$\Ext^q_{\qgr A_\e}(\cU, \cU \otimes_{\cA_\e}(\omega _{\cA_\e}^{-1})^{\otimes i})=0$ for $q \neq 0$ and $i \geq 0$.
In particular, it follows that $(\cA_\e(-n+3), \dots, \cA_\e, \cX_1, \dots, \cX_\al)$ is strong.
By \lemref{lem.2c},
$(\cU, -\otimes_{\cA_\e}\omega _{\cA_\e}^{-1})$ is ample for $\qgr A_\e$.
Therefore, \thmref{thm.MUq} yields the desired conclusions.
\end{proof}

\begin{rem}
If $n=1$, then $A_\e$ is finite-dimensional algebra, so $\qgr A_\e$ is trivial.
If $n=2$, then one can check that $\Db(\qgr A_\e) \cong \Db(\coh \Proj k[x,y]/(x^2+y^2)) \cong \Db(\mod k^2)$,
so $\Db(\qgr A_\e)$ has a full strong exceptional sequence (in this case, $\La_\e \cong k^2$ and 
$\Pi\La_\e \cong \La_\e[X] \cong k[X]^{2}$ where $\deg X=1$).
\end{rem}

\subsection{Quiver Presentations}

We here calculate the quiver presentation of $\La_\e= \End_{\qgr A_{\e}}(\cU)$ in \thmref{thm.main}
where $\cU= A_\e(-n+3) \oplus \cdots \oplus \cA_\e \oplus \cX_1 \oplus \cdots \oplus \cX_\al$.
For simplicity, we write $A:=A_\e$.
We see that $\La_\e$ is isomorphic to
\renewcommand{\arraycolsep}{0.75mm}
{\scriptsize
\begin{align*} 
\begin{pmatrix}
\End(\cA(c)) &0   &\cdots &0 &0 &0&0&\cdots &0&0\\
[\cA(c),\cA(c+1)]&\End(\cA(c+1)) &\cdots &0 &0 &0&0 &\cdots &0&0\\
\vdots&\vdots &\ddots &\vdots &\vdots &\vdots &\vdots &\cdots &\vdots&\vdots\\
[\cA(c),\cA(-1)] &[\cA(c+1),\cA(-1)]&\cdots &\End(\cA(-1)) &0 &0 &0&\cdots &0&0\\
[\cA(c),\cA] &[\cA(c+1),\cA]&\cdots &[\cA(-1),\cA] &\End(\cA) & 0& 0&\cdots &0&0\\
[\cA(c),\cX_1] &[\cA(c+1),\cX_1]&\cdots &[\cA(-1),\cX_1] &[\cA,\cX_1] &\End(\cX_1) &0 &\cdots &0&0\\
[\cA(c),\cX_2] &[\cA(c+1),\cX_2]&\cdots &[\cA(-1),\cX_2] &[\cA,\cX_2] &[\cX_1,\cX_2]&\End(\cX_2)&\cdots &0&0\\
\vdots&\vdots &\vdots &\vdots &\vdots &\vdots &\vdots &\ddots &\vdots &\vdots\\
[\cA(c),\cX_{\al-1}] &[\cA(c+1),\cX_{\al-1}]&\cdots &[\cA(-1),\cX_{\al-1}] &[\cA,\cX_{\al-1}] &[\cX_1,\cX_{\al-1}] &[\cX_2,\cX_{\al-1}] &\cdots
&\End(\cX_{\al-1}) &0\\
[\cA(c),\cX_{\al}] &[\cA(c+1),\cX_{\al}]&\cdots &[\cA(-1),\cX_{\al}] &[\cA,\cX_{\al}] &[\cX_1,\cX_{\al}] &[\cX_2,\cX_{\al}]&\cdots &[\cX_{\al-1},\cX_{\al}] &\End(\cX_{\al})
\end{pmatrix}
\end{align*}
}where we put $c=-n+3$, $[\cM,\cN]= \Hom_{\qgr A}(\cM,\cN)$ and $\End(\cM)= \End_{\qgr A}(\cM)$ for short.
By \lemref{lem.HE} (1), it is isomorphic to
{\small
\begin{align*}
\begin{pmatrix}
A_0 &0   &\cdots &0 &0 &0&0&\cdots &0&0\\
A_1 &A_0  &\cdots &0 &0 &0&0 &\cdots &0&0\\
\vdots&\vdots &\ddots &\vdots &\vdots &\vdots &\vdots &\cdots &\vdots&\vdots\\
A_{n-4} &A_{n-5}&\cdots &A_0  &0 &0 &0&\cdots &0&0\\
A_{n-3} &A_{n-4}&\cdots &A_1 &A_0  & 0& 0&\cdots &0&0\\
(X_1)_{n-3} &(X_1)_{n-4}&\cdots &(X_1)_{1} &(X_1)_{0} &\End(X_1) &0 &\cdots &0&0\\
(X_2)_{n-3} &(X_2)_{n-4}&\cdots &(X_2)_{1} &(X_2)_{0} &0 &\End(X_2)&\cdots &0&0\\
\vdots&\vdots &\vdots &\vdots &\vdots &\vdots &\vdots &\ddots &\vdots &\vdots\\
(X_{\al-1})_{n-3} &(X_{\al-1})_{n-4}&\cdots &(X_{\al-1})_{1} &(X_{\al-1})_{0} &0 &0 &\cdots &\End(X_{\al-1}) &0\\
(X_{\al})_{n-3} &(X_{\al})_{n-4}&\cdots &(X_{\al})_{1} &(X_{\al})_{0} &0 &0&\cdots &0 &\End(X_{\al})
\end{pmatrix}
\end{align*}
}where we put $\End(M) = \End_{\grmod A}(M)$ for short.
Note that $A_0 \cong \End(X_i)\cong k$.

\begin{thm}\label{thm.quiv}
In the situation of \thmref{thm.main}, $\La_\e$ is presented by the quiver
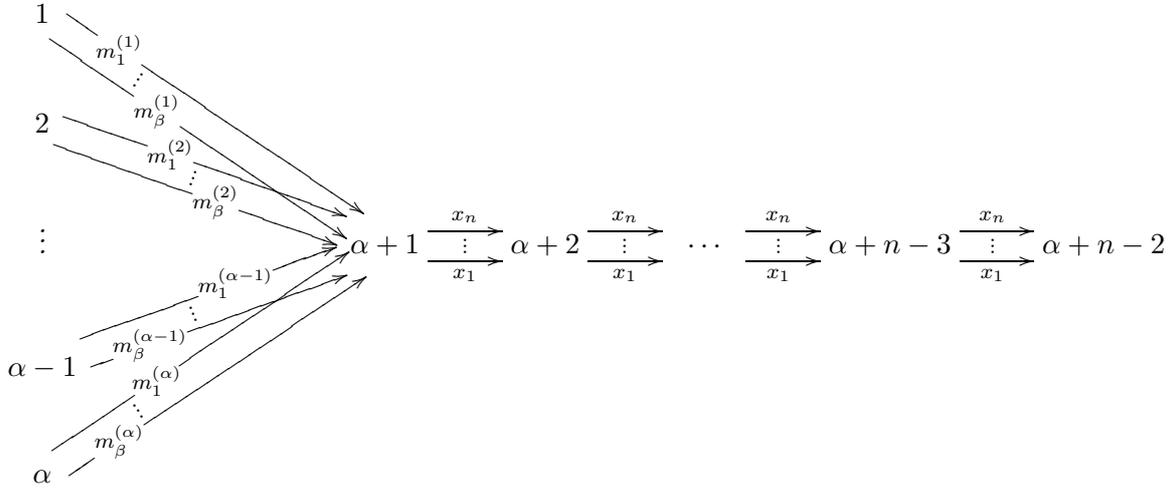
\begin{align*}
\xymatrix @R=2.3pc@C=2.3pc{
1 \ar@<1.2ex>[rrrdd]|(0.2){m_1^{(1)}} \ar@{}[rrrdd]|(0.28){\scriptstyle \cdot} \ar@{}@<0.5ex>[rrrdd]|(0.28){\scriptstyle \cdot} \ar@{}@<-0.5ex>[rrrdd]|(0.28){\scriptstyle \cdot} \ar@<-1.2ex>[rrrdd]|(0.36){m_\be^{(1)}} \\
2 \ar@<1.2ex>[rrrd]|(0.36){m_1^{(2)}} \ar@{}[rrrd]|(0.44){\scriptstyle \cdot} \ar@{}@<0.5ex>[rrrd]|(0.44){\scriptstyle \cdot} \ar@{}@<-0.5ex>[rrrd]|(0.44){\scriptstyle \cdot} \ar@<-1.2ex>[rrrd]|(0.52){m_\be^{(2)}} \\
\vdots
&&&\al+1
\ar@<-1.2ex>[r]_{x_1} \ar@{}[r]|{\scriptstyle \cdot} \ar@{}@<-0.5ex>[r]|{\scriptstyle \cdot} \ar@{}@<0.5ex>[r]|{\scriptstyle \cdot} \ar@<1.2ex>[r]^{x_n}
&\al+2 \ar@<-1.2ex>[r]_{x_1} \ar@{}[r]|{\scriptstyle \cdot} \ar@{}@<-0.5ex>[r]|{\scriptstyle \cdot} \ar@{}@<0.5ex>[r]|{\scriptstyle \cdot} \ar@<1.2ex>[r]^{x_n}
&\;\;\cdots\;\;
\ar@<-1.2ex>[r]_(0.4){x_1} \ar@{}[r]|(0.4){\scriptstyle \cdot} \ar@{}@<-0.5ex>[r]|(0.4){\scriptstyle \cdot} \ar@{}@<0.5ex>[r]|(0.4){\scriptstyle \cdot} \ar@<1.2ex>[r]^(0.4){x_n}
&\al+n-3 \ar@<-1.2ex>[r]_(0.48){x_1} \ar@{}[r]|(0.48){\scriptstyle \cdot} \ar@{}@<-0.5ex>[r]|(0.48){\scriptstyle \cdot} \ar@{}@<0.5ex>[r]|(0.48){\scriptstyle \cdot} \ar@<1.2ex>[r]^(0.48){x_n}
&\al+n-2 \\
\al-1\ar@<1.2ex>[rrru]|(0.58){m_1^{(\al-1)}} \ar@{}[rrru]|(0.44){\scriptstyle \cdot} \ar@{}@<0.5ex>[rrru]|(0.44){\scriptstyle \cdot} \ar@{}@<-0.5ex>[rrru]|(0.44){\scriptstyle \cdot} \ar@<-1.2ex>[rrru]|(0.3){m_\be^{(\al-1)}}\\
\al \ar@<1.2ex>[rrruu]|(0.36){m_1^{(\al)}} \ar@{}[rrruu]|(0.28){\scriptstyle \cdot} \ar@{}@<0.5ex>[rrruu]|(0.28){\scriptstyle \cdot} \ar@{}@<-0.5ex>[rrruu]|(0.28){\scriptstyle \cdot} \ar@<-1.2ex>[rrruu]|(0.2){m_\be^{(\al)}}
}
\end{align*}
with relations
\begin{align*}
&x_ix_j - \ep_{ij}x_jx_i=0,\\
&x_1^2+\cdots+x_n^2=0, \\
&g^{(s,t)}:= \sum_{\substack{1 \leq i \leq \be \\ 1\leq j\leq n }} \la_{i,j}^{(s,t)} m_{i}^{(s)}x_{j}=0 \quad (s=1,2,\dots,\al,\; t=1,2,\dots,\be)
\end{align*}
for some constants $\la_{ij}^{(s,t)} \in k$.
\end{thm}

\begin{proof}
It follows from the matrix description of $\La_\e$ above that the shape of the quiver is as claimed
(note that the idempotent $e_{ii}$ of the matrix description corresponds to the vertex $\al +n-1-i$ of the quiver, where $e_{ii}$ is the matrix such that the $(i, i)$-entry is $1$ and the other entries are all $0$).

Now let us prove the relations.  Clearly, $x_ix_j - \ep_{ij}x_jx_i=0$ and 
$x_1^2+\cdots+x_n^2=0$ are satisfied in $\La_\e$.
By \lemref{lem.x} (3), there exists an exact sequence
\[\xymatrix @R=1.5pc@C=1.5pc{
0 \ar[r] & \Ker\vph_s \ar[r] &A^{\be} \ar[r]^{\vph_s} & X_s \ar[r] &0
}\] 
in $\grmod A$ for any $1\leq s \leq \al$,
so we can take $m_{1}^{(s)},\dots, m_{\be}^{(s)} \in (X_s)_0$ such that they form a minimal set of generators of $X_s$,
and $\vph_s$ is the left multiplication of $\begin{pmatrix} m_{1}^{(s)} &\cdots &m_{\be}^{(s)} \end{pmatrix}$. 
By \lemref{lem.x} (3), 
it follows that $\Ker\vph_s$ is isomorphic to $X_{s'}(-1)$ for some $s'$.
Thus we can take $\la_{i,j}^{(s,t)} \in k$ such that
\[ \boldsymbol{a}^{(s,1)} =\begin{pmatrix} \sum_{1\leq j\leq n}\la_{1,j}^{(s,1)}x_j \\ \vdots \\ \sum_{1\leq j\leq n}\la_{\be,j}^{(s,1)}x_j \end{pmatrix},
\dots,
\boldsymbol{a}^{(s,\be)} =\begin{pmatrix} \sum_{1\leq j\leq n}\la_{1,j}^{(s,\be)}x_j \\ \vdots \\ \sum_{1\leq j\leq n}\la_{\be,j}^{(s,\be)}x_j \end{pmatrix}
\; \in (\Ker\vph_s)_0 \; \subset (A^\be)_1
\]
form a minimal set of generators of $\Ker\vph_s$.
For any $1 \leq t \leq \be$, we have
\begin{align*}
g^{(s,t)}&=\sum_{\substack{1 \leq i \leq \be \\ 1\leq j\leq n }} \la_{i,j}^{(s,t)}m_{i}^{(s)}x_j
= m_{1}^{(s)}(\sum_{1\leq j\leq n}\la_{1,j}^{(s,t)}x_j)+\cdots +m_{\be}^{(s)}(\sum_{1\leq j\leq n}\la_{\be,j}^{(s,t)}x_j)\\
&=\begin{pmatrix} m_{1}^{(s)} &\cdots &m_{\be}^{(s)} \end{pmatrix}\begin{pmatrix} \sum_{1\leq j\leq n}\la_{1,j}^{(s,t)}x_j \\ \vdots \\ \sum_{1\leq j\leq n}\la_{\be,j}^{(s,t)}x_j \end{pmatrix}
=\vph_s(\boldsymbol{a}^{(s,t)})=0
\end{align*}
in $(X_s)_1$. Thus $g^{(s,t)}=0$ is satisfied in $\La_\e$.

To show no more relation is needed, suppose that $m_{1}^{(s)}a_1+\cdots+m_{\be}^{(s)}a_\be =0$ in $(X_s)_i$ for some $a_1,\dots, a_\be \in A_i$.
Then ${\scriptstyle \begin{pmatrix} a_{1} \\ \vdots \\ a_{\be} \end{pmatrix}}$ is in $\Ker \vph_s$, so it is generated by
$\boldsymbol{a}^{(s,1)}, \dots , \boldsymbol{a}^{(s,\be)}$, i.e.
\[
\begin{pmatrix} a_{1} \\ \vdots \\ a_{\be} \end{pmatrix} = \sum_{t=1}^{\be}\boldsymbol{a}^{(s,t)}c_t.
\]
Hence we obtain
\[
m_{1}^{(s)}a_1+\cdots+m_{\be}^{(s)}a_\be = \sum_{t=1}^{\be}g^{(s,t)}c_t,
\]
and this finishes the proof.
\end{proof}

\begin{rem}
If we have minimal free resolutions of $X_1, \dots, X_\al \in \CM(A_\e)$ explicitly, in other words, if we have indecomposable reduced matrix factorizations of $f_\e$ explicitly, then we can give coefficients $\la_{ij}^{(s,t)}$ explicitly (see \exref{ex.-1}, Section 4).
Note that $\la_{ij}^{(s,t)}$ depends on the choice of generators of $X_i$.
The author does not know whether the corresponding algebras are isomorphic, when we choose different generators and give different coefficients.
\end{rem}

Typical examples are the following.

\begin{ex}\label{ex.n3}
Assume that $n=3$.
Then one of the following two cases occurs.

(1) If $\e_{12}\e_{13}\ep_{23}=1$, then one can verify $r=\mnull_{\F_2} \De_{\e}=0$. In this case, $\al=2^0=1 ,\be =2^1=2$, so
$\La_\e$ is isomorphic to the path algebra $kQ$ of the quiver  
\[ Q = \xymatrix@R=1pc@C=2.5pc{
1 \ar@<0.5ex>[r]^(0.5){m_1^{(1)}} \ar@<-0.5ex>[r]_(0.5){m_2^{(1)}} &2
}\quad (\text{$\widetilde{A_1}$ type}).\] 

(2) If $\e_{12}\e_{13}\ep_{23}=-1$, then one can verify $r=\mnull_{\F_2} \De_{\e}=2$. In this case, $\al=2^2=4 ,\be =2^0=1$, so
$\La_\e$ is isomorphic to the path algebra $kQ$ of the quiver
\[Q =\vcenter{
\xymatrix@R=0.25pc@C=2.5pc{
1 \ar[rd]^(0.4){m_1^{(1)}} &&4 \ar[ld]_(0.4){m_1^{(4)}}\\
 &5\\
2 \ar[ru]_(0.4){m_1^{(2)}} &&3 \ar[lu]^(0.4){m_1^{(3)}}
}}\quad (\text{$\widetilde{D_4}$ type}).\] 
\end{ex}

\begin{ex}\label{ex.-1}
Assume that $\e_{ij}=-1$ for every $1 \leq i<j \leq n$. Then it is easy to see $r=\mnull_{\F_2} \De_{\e}=n-1$.
Thus $\al=2^{n-1} ,\be =2^0=1$.

For $1 \leq s \leq 2^{n-1}=\al$, we represent
\[ s-1 = c^{(s)}_1c^{(s)}_2\cdots c^{(s)}_{n-1}\]
as a binary number where $c^{(s)}_1, \dots, c^{(s)}_{n-1} \in \{0,1\}$. We define
\begin{align*}
&a^{(s)} := x_1 + (-1)^{c^{(s)}_1}x_2 + (-1)^{c^{(s)}_2}x_3 + \cdots + (-1)^{c^{(s)}_{n-1}}x_n \; \in A_\e,\\
&X_s := {A_\e}/a^{(s)}{A_\e}.
\end{align*}
Then $X_1, \dots, X_\al \in \M_\e$. (Note that $X_s$ has a minimal free resolution
\[\xymatrix @R=1pc@C=2pc{
\cdots \ar[r]^(0.4){a^{(s)}\cdot} & A_\e(-2) \ar[r]^{a^{(s)}\cdot} & A_\e(-1) \ar[r]^(0.6){a^{(s)}\cdot} & A_\e \ar[r] & X_s \ar[r] & 0
}\]
in $\grmod A_\e$.) Hence $\La_\e$ is given by the quiver
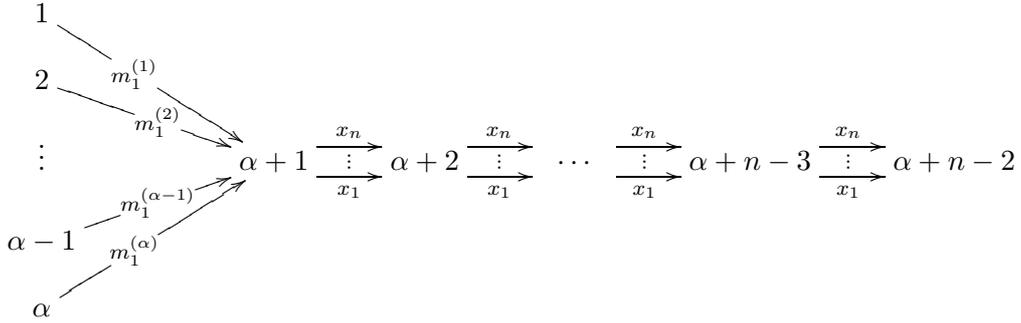
\begin{align*}
\xymatrix @R=1pc@C=2pc{
1 \ar[rrdd]|(0.4){m_1^{(1)}} \\
2 \ar[rrd]|(0.5){m_1^{(2)}} \\
\vdots
&&\al+1
\ar@<-1.2ex>[r]_{x_1} \ar@{}[r]|{\scriptstyle \cdot} \ar@{}@<-0.5ex>[r]|{\scriptstyle \cdot} \ar@{}@<0.5ex>[r]|{\scriptstyle \cdot} \ar@<1.2ex>[r]^{x_n}
&\al+2 \ar@<-1.2ex>[r]_{x_1} \ar@{}[r]|{\scriptstyle \cdot} \ar@{}@<-0.5ex>[r]|{\scriptstyle \cdot} \ar@{}@<0.5ex>[r]|{\scriptstyle \cdot} \ar@<1.2ex>[r]^{x_n}
&\;\;\cdots\;\;
\ar@<-1.2ex>[r]_(0.4){x_1} \ar@{}[r]|(0.4){\scriptstyle \cdot} \ar@{}@<-0.5ex>[r]|(0.4){\scriptstyle \cdot} \ar@{}@<0.5ex>[r]|(0.4){\scriptstyle \cdot} \ar@<1.2ex>[r]^(0.4){x_n}
&\al+n-3 \ar@<-1.2ex>[r]_(0.48){x_1} \ar@{}[r]|(0.48){\scriptstyle \cdot} \ar@{}@<-0.5ex>[r]|(0.48){\scriptstyle \cdot} \ar@{}@<0.5ex>[r]|(0.48){\scriptstyle \cdot} \ar@<1.2ex>[r]^(0.48){x_n}
&\al+n-2 \\
\al-1\ar[rru]|(0.5){m_1^{(\al-1)}}\\
\al \ar[rruu]|(0.4){m_1^{(\al)}}
}
\end{align*}
with relations
\begin{align*}
&x_ix_j + x_jx_i=0 \quad (i\neq j),\\
&x_1^2+\cdots+x_n^2=0, \\
&g^{(s)}:= \sum_{1\leq j \leq n} (-1)^{c^{(s)}_{j-1}}m_1^{(s)}x_j =0 \quad (s=1,2,\dots,\al)
\end{align*}
where we consider $c^{(s)}_0 = 0$.
\end{ex}

Incidentally, combining our results and a version of the BGG correspondence \cite[Theorem 5.3 and Lemma 5.1]{Mrr},
we have the following equivalences of categories.

\begin{cor} \label{cor.dim1} 
Let $B_\e= k\ang{x_1, \dots, x_n}/(\e_{ij}x_ix_j+x_jx_i, x_i^2-x_n^2)_{1\leq i,j\leq n,i\neq j}$ with $\deg x_i=1$.
Then $B_\e$ is a noetherian AS-Gorenstein Koszul algebra of dimension $1$
such that $B_\e^! \cong A_\e$.
If $n  \geq 3$ and $\La_\e$ is as in \thmref{thm.main} (or \thmref{thm.quiv}), then we have an equivalence of triangulated categories
\[ \uCM(B_\e) \cong \Db(\qgr A_\e) \cong \Db(\mod \La_\e). \]
\end{cor}

\begin{rem}
This result can be seen as an analogue of Buchweitz-Iyama-Yamaura's theorem \cite[Theorem 1.4]{BIY} on 
graded commutative Gorenstein rings of dimension $1$.
\end{rem}

\section{Classification for $n=4$}
The classification of $\Db(\qgr A_\e)$ for $n=3$ is obtained in \exref{ex.n3}.
In this section, we give the classification for $n=4$.
We continue to use Notation \ref{nota}. 

First, using the classification of the graphs $G_\e$ up to mutation and isomorphism,
we explain that if $n=4$, then $\qgr A_\e$ can be divided into three cases.

\begin{lem}\label{lem.Zh}
Let $G_{\e}, G_{\e'}$ be the graphs associated with $\e, \e'$. 
If $G_{\e'}$ is obtained from $G_{\e}$ by a mutation, then $\GrMod A_\e \cong \GrMod A_{\e'}$ and $\qgr A_{\e} \cong \qgr A_{\e'}$.
\end{lem}

\begin{proof}
By reordering the vertices, without loss of generality, we may assume that $G_{\e'}$ is obtained from $G_{\e}$ by mutation at $n$.
Let $\theta$ be the graded algebra automorphism of $A_\e$ defined by $\theta(x_i)=x_i$ for $i \neq n$ and $\theta(x_n)=-x_n$. Then the twisted algebra of $A_\e$ by the twisting system $\{\theta^i \}_{i \in \Z}$ in the sense of Zhang \cite{Zh}  is isomorphic to
$S_{\e'}/(x_1^2+\cdots+x_{n-1}^2-x_n^2)$.
Moreover, the algebra homomorphism 
$\psi: S_{\e'}/(x_1^2+\cdots+x_{n-1}^2-x_n^2) \to  A_{\e'}$
defined by 
$\psi(x_i)= x_i$ for $i \neq n$ and $\psi(x_n)=\sqrt{-1}x_n$ is an isomorphism.
Thus one gets 
\[\GrMod A_\e \cong \GrMod  S_{\e'}/(x_1^2+\cdots+x_{n-1}^2-x_n^2) \cong \GrMod  A_{\e'} \]
by \cite[Theorem 1.1]{Zh}. This equivalence induces  $\qgr A_{\e} \cong \qgr A_{\e'}$
(see \cite[Section 4.1]{Ro} for example).
\end{proof}

From now on, we focus on the case $n=4$.  It is easy to see that every $G_{\e}$ becomes one of the following graphs by iterated mutations up to isomorphism:
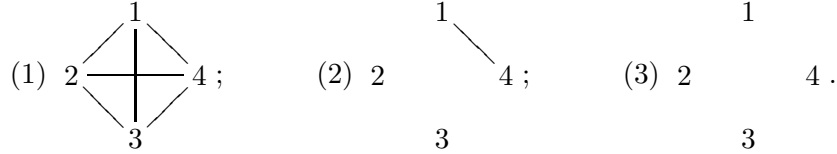
\[(1)\; \xy /r2pc/: 
{\xypolygon4{~={90}~*{\xypolynode}~>{}}},
"1";"2"**@{-},
"1";"3"**@{-},
"1";"4"**@{-},
"2";"3"**@{-},
"2";"4"**@{-},
"3";"4"**@{-},
\endxy;
\qquad\quad 
(2) \; \xy /r2pc/: 
{\xypolygon4{~={90}~*{\xypolynode}~>{}}},
"1";"4"**@{-},
\endxy;
\qquad\quad 
(3) \; \xy /r2pc/: 
{\xypolygon4{~={90}~*{\xypolynode}~>{}}},
\endxy.
\]
Therefore, by \lemref{lem.Zh}, we have the following.

\begin{prop} \label{prop.GM}
If $n=4$, then any $\qgr A_{\e'}$ (resp. $\GrMod A_{\e'}$) is equivalent to
$\qgr A_\e$ (resp. $\GrMod A_{\e}$) where $A_\e$ is one of the following algebras:
\begin{enumerate}
\item[(QS1)] $A_\e = k[x, y, z, w]/(x^2+y^2+z^2+w^2)$;
\item[(QS2)] $A_\e =  k\<x, y, z, w\>/(
xy+yx,\ xz+zx,\ xw-wx,\ zy+yz,\ wy+yw,\ zw+wz,\ x^2+y^2+z^2+w^2)$;
\item[(QS3)] $A_\e =  k\<x, y, z, w\>/(
xy+yx,\ xz+zx,\ xw+wx,\ zy+yz,\ wy+yw,\ zw+wz,\ x^2+y^2+z^2+w^2)$.
\end{enumerate}
\end{prop}

We next show that if $i \neq j$ then (QS$i$) and (QS$j$) are not derived equivalent.
To do this, we compute the Hochschild cohomology $\HH^i(\La_\e)$ of $\La_\e$.

\begin{lem} \label{lem.HH}
If $\dim_k \HH^i(\La_\e) \neq \dim_k \HH^i(\La_{\e'})$ for some $i$, then 
$\Db(\qgr A_{\e}) \not\cong \Db(\qgr A_{\e'})$. 
\end{lem}

\begin{proof}
We show the contraposition.
If $\Db(\qgr A_{\e}) \cong \Db(\qgr A_{\e'})$, then we have
\[ \Db(\mod \La_{\e})\cong  \Db(\qgr A_{\e})\cong \Db(\qgr A_{\e'}) \cong \Db(\mod \La_{\e'}),\]
so it follows that $\dim_k \HH^i(\La_\e)=\dim_k \HH^i(\La_{\e'})$ for any $i$
by \cite[Proposition 2.5]{Ri}.
\end{proof}

Let $\La$ be a finite-dimensional algebra of the form $kQ/I$ where $Q$ is a finite quiver
and $I$ is an admissible ideal.
We denote by $G^0$ the set of all vertices of $Q$, by $G^1$ the set of all arrows of $Q$,
and by $G^2$ a minimal set of uniform generators of $I$.
For $h \in G^i$, we write $s(h)$ for the source and $t(h)$ for the target.
We now assume that $\gldim \La=2$. 
Then $\projdim_{\La^{\en}} \La = 2$, so by \cite[Theorem 2.9]{GS},
we can construct explicitly a minimal projective resolution of $\La$ 
\begin{align} \label{mpr}
0\longrightarrow P^{2}\stackrel{A^{2}}{\longrightarrow}
P^{1} \stackrel{A^{1}}{\longrightarrow}
P^{0}\stackrel{\partial}{\longrightarrow} \La
\longrightarrow 0
\end{align}
as a right $\La^{\en}$-module
where 
\begin{align*}
P^{i}:=\bigoplus_{h \in G^{i}}\La s(h)\otimes t(h)\La
\end{align*}
for $i=0,1,2$.
See \cite[Section 2]{GS} for details of the constructions of $A^1$ and $A^2$.
By applying $ (-)^*:=\Hom_{\La^{\rm e}}(-,\La)$ to (\ref{mpr}),
we have the Hochschild complex 
\begin{align} \label{hcpx}
0 \longrightarrow (P^{0})^* \stackrel{(A^1)^*}{\longrightarrow} 
  (P^{1})^* \stackrel{(A^2)^*}{\longrightarrow} 
  (P^{2})^* \longrightarrow 0. 
\end{align}
The Hochschild cohomology $\HH^i(\La)$ can be described by the cohomology of this complex.
Clearly, $\HH^{i}(\La)=0$ for $i\geq 3$.
Moreover, it is known that if the quiver $Q$ is connected and has no oriented cycles,
then $\dim_k \HH^0(\La) =1$,
and hence in order to calculate $\dim_k \HH^1(\La)$ and $\dim_k \HH^2(\La)$,
it is enough to calculate $\dim_k  (P^{0})^*, \dim_k  (P^{1})^*, \dim_k  (P^{2})^*$ and $\dim_k \Im (A^2)^*$.

\begin{dfn}
\begin{enumerate}
\item For $\ga$ in $G^0$, we define the right $\La^{\en}$-homomorphism $\th_{\ga} \in  (P^{0})^*$ by
\[
\th_{g}(s(h)\otimes t(h))=
\begin{cases}
g & \text{if }h=\ga,\\
0 & \text{otherwise.}
\end{cases}
\]
\item For $u, \ga$ in $G^1$ with $s(u)=s(\ga), t(u)=t(\ga)$, we define the right $\La^{\en}$-homomorphism $\th_{u, \ga} \in  (P^{1})^*$ by
\[
\th_{u, \ga}(s(h)\otimes t(h))=
\begin{cases}
u & \text{if }h=\ga,\\
0 & \text{otherwise.}
\end{cases}
\]

\item For a path $u$ in $Q$ and $\ga$ in $G^2$ with $s(u)=s(\ga), t(u)=t(\ga)$, we define the right $\La^{\en}$-homomorphism $\th_{u, \ga} \in  (P^{2})^*$ by
\[
\th_{u, \ga}(s(h)\otimes t(h))=
\begin{cases}
u & \text{if }h=\ga,\\
0 & \text{otherwise.}
\end{cases}
\]
\end{enumerate}
\end{dfn}

\noindent \textbf{The case \textrm{ (QS1)}.}
Now, let us consider the case (QS1).
In this case, $ A_\e =S_\e/(f_\e)$ where
$S_\e=k[x,y,z,w]$ and $f_\e= x^2+y^2+z^2+w^2$.
Since $\mnull_{\F_2} \De_{\e}=1$, we have $\al =|\M_\e|=2$
and therefore
$\uCM(A) \cong \Db(\mod k^2)$.
It is easy to see that
\[
\left(
\begin{pmatrix}
x+iw &y+iz \\
y-iz &-x+iw
\end{pmatrix},
\begin{pmatrix}
x-iw &y+iz \\
y-iz &-x-iw
\end{pmatrix}
\right)
\;\text{and}\;
\left(
\begin{pmatrix}
x-iw &y+iz \\
y-iz &-x-iw
\end{pmatrix},
\begin{pmatrix}
x+iw &y+iz \\
y-iz &-x+iw
\end{pmatrix}
\right)
\]
are matrix factorizations of $f_\e$ in $S_\e$ where $i:=\sqrt{-1}$.
Thus 
\[X_1 := \Coker \begin{pmatrix}
x+iw &y+iz \\
y-iz &-x+iw
\end{pmatrix}\cdot ,\;
X_2 := \Coker \begin{pmatrix}
x-iw &y+iz \\
y-iz &-x-iw
\end{pmatrix}\cdot \;
\in \M_\e. \]
Hence $\La_\e$ is given by the quiver
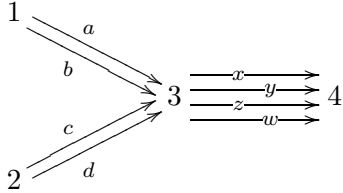
\begin{align*}
\xymatrix @R=1.5pc@C=4pc{
1 \ar@<0.5ex>[rd]^(0.4){a} \ar@<-0.5ex>[rd]_(0.4){b} \\
&3
\ar@<-1.8ex>[r]|(0.6){w}
\ar@<-0.6ex>[r]|(0.4){z}
\ar@<0.6ex>[r]|(0.6){y}
\ar@<1.8ex>[r]|(0.4){x}
&4 \\
2 \ar@<0.5ex>[ru]^(0.4){c} \ar@<-0.5ex>[ru]_(0.4){d}
}
\end{align*}
with relations
\begin{align*}
&g_1:=ax+iaw+by-ibz=0,
&&g_2:=ay+iaz-bx+ibw=0,\\
&g_3:=cx-icw+dy-idz=0,
&&g_4:=cy+icz-dx-idw=0.
\end{align*}
Since $\gldim \La_\e= 2$,
we can construct the Hochschild complex (\ref{hcpx}).
Since the quiver is connected and has no oriented cycles,
it follows that $\dim_k \HH^0(\La_\e) =1$.
One can verify that $\dim_k  (P^{0})^*=4, \dim_k  (P^{1})^*=24, \dim_k  (P^{2})^*=24$ and $\dim_k \Im (A^2)^*=15$.
In fact, $(P^{0})^*$ has a basis $\{\th_{e_1},\dots,\th_{e_4}\}$,
$(P^{1})^*$ has an ordered basis
\begin{align*}\Theta^1 = 
(&\th_{a,a},\th_{b,a},\th_{a,b},\th_{b,b},
\th_{c,c},\th_{d,c},\th_{c,d},\th_{d,d},
\th_{x,x},\th_{y,x},\th_{z,x},\th_{w,x},\\
&\th_{x,y},\th_{y,y},\th_{z,y},\th_{w,y},
\th_{x,z},\th_{y,z},\th_{z,z},\th_{w,z},
\th_{x,w},\th_{y,w},\th_{z,w},\th_{w,w}
),
\end{align*}
and $(P^{2})^*$ has an ordered basis 
\begin{align*}\Theta^2 = 
(&\th_{ay,g_1},\th_{iaz,g_1},\th_{iaw,g_1},
\th_{by,g_1},\th_{ibz,g_1},\th_{ibw,g_1},
\th_{ay,g_2},\th_{iaz,g_2},\th_{iaw,g_2},
\th_{by,g_2},\th_{ibz,g_2},\th_{ibw,g_2},\\
&\th_{cy,g_3},\th_{icz,g_3},\th_{icw,g_3},
\th_{dy,g_3},\th_{idz,g_3},\th_{idw,g_3},
\th_{cy,g_4},\th_{icz,g_4},\th_{icw,g_4},
\th_{dy,g_4},\th_{idz,g_4},\th_{idw,g_4}
).
\end{align*}
Moreover, one can verify that the matrix representation of $(A^2)^*$ with respect to $\Theta^1$ and $\Theta^2$ is
\footnotesize
\[\scriptstyle \left(\begin{array}{cccccccccccccccccccccccc} 
0 & 1 & 1 & 0 & 0 & 0 & 0 & 0 & 0 & 1 & 0 & 0 & 1 & 0 & 0 & 0 & -1 & 0 & 0 & 0 & 0 & 1 & 0 & 0\\
0 & 1 & -1 & 0 & 0 & 0 & 0 & 0 & 0 & 0 & 1 & 0 & 1 & 0 & 0 & 0 & -1 & 0 & 0 & 0 & 0 & 0 & 1 & 0\\
0 & 0 & 0 & 0 & 0 & 0 & 0 & 0 & -1 & 0 & 0 & 1 & 0 & 0 & 0 & 0 & 0 & 0 & 0 & 0 & -1 & 0 & 0 & 1\\
-1 & 0 & 0 & 1 & 0 & 0 & 0 & 0 & -1 & 0 & 0 & 0 & 0 & 1 & 0 & 0 & 0 & -1 & 0 & 0 & -1 & 0 & 0 & 0\\
1 & 0 & 0 & -1 & 0 & 0 & 0 & 0 & 1 & 0 & 0 & 0 & 0 & 0 & 1 & 0 & 0 & 0 & -1 & 0 & 1 & 0 & 0 & 0\\
0 & 2 & 0 & 0 & 0 & 0 & 0 & 0 & 0 & 0 & 0 & 0 & 1 & 0 & 0 & 1 & -1 & 0 & 0 & -1 & 0 & 0 & 0 & 0\\
1 & 0 & 0 & -1 & 0 & 0 & 0 & 0 & -1 & 0 & 0 & 0 & 0 & 1 & 0 & 0 & 0 & 1 & 0 & 0 & 1 & 0 & 0 & 0\\
1 & 0 & 0 & -1 & 0 & 0 & 0 & 0 & -1 & 0 & 0 & 0 & 0 & 0 & 1 & 0 & 0 & 0 & 1 & 0 & 1 & 0 & 0 & 0\\
0 & 0 & 2 & 0 & 0 & 0 & 0 & 0 & 0 & 0 & 0 & 0 & -1 & 0 & 0 & 1 & -1 & 0 & 0 & 1 & 0 & 0 & 0 & 0\\
0 & 1 & 1 & 0 & 0 & 0 & 0 & 0 & 0 & -1 & 0 & 0 & -1 & 0 & 0 & 0 & -1 & 0 & 0 & 0 & 0 & 1 & 0 & 0\\
0 & 1 & -1 & 0 & 0 & 0 & 0 & 0 & 0 & 0 & -1 & 0 & 1 & 0 & 0 & 0 & 1 & 0 & 0 & 0 & 0 & 0 & 1 & 0\\
0 & 0 & 0 & 0 & 0 & 0 & 0 & 0 & -1 & 0 & 0 & -1 & 0 & 0 & 0 & 0 & 0 & 0 & 0 & 0 & 1 & 0 & 0 & 1\\
0 & 0 & 0 & 0 & 0 & 1 & 1 & 0 & 0 & 1 & 0 & 0 & 1 & 0 & 0 & 0 & -1 & 0 & 0 & 0 & 0 & -1 & 0 & 0\\
0 & 0 & 0 & 0 & 0 & 1 & -1 & 0 & 0 & 0 & 1 & 0 & 1 & 0 & 0 & 0 & -1 & 0 & 0 & 0 & 0 & 0 & -1 & 0\\
0 & 0 & 0 & 0 & 0 & 0 & 0 & 0 & 1 & 0 & 0 & 1 & 0 & 0 & 0 & 0 & 0 & 0 & 0 & 0 & -1 & 0 & 0 & -1\\
0 & 0 & 0 & 0 & -1 & 0 & 0 & 1 & -1 & 0 & 0 & 0 & 0 & 1 & 0 & 0 & 0 & -1 & 0 & 0 & 1 & 0 & 0 & 0\\
0 & 0 & 0 & 0 & 1 & 0 & 0 & -1 & 1 & 0 & 0 & 0 & 0 & 0 & 1 & 0 & 0 & 0 & -1 & 0 & -1 & 0 & 0 & 0\\
0 & 0 & 0 & 0 & 0 & -2 & 0 & 0 & 0 & 0 & 0 & 0 & -1 & 0 & 0 & 1 & 1 & 0 & 0 & -1 & 0 & 0 & 0 & 0\\
0 & 0 & 0 & 0 & 1 & 0 & 0 & -1 & -1 & 0 & 0 & 0 & 0 & 1 & 0 & 0 & 0 & 1 & 0 & 0 & -1 & 0 & 0 & 0\\
0 & 0 & 0 & 0 & 1 & 0 & 0 & -1 & -1 & 0 & 0 & 0 & 0 & 0 & 1 & 0 & 0 & 0 & 1 & 0 & -1 & 0 & 0 & 0\\
0 & 0 & 0 & 0 & 0 & 0 & -2 & 0 & 0 & 0 & 0 & 0 & 1 & 0 & 0 & 1 & 1 & 0 & 0 & 1 & 0 & 0 & 0 & 0\\
0 & 0 & 0 & 0 & 0 & 1 & 1 & 0 & 0 & -1 & 0 & 0 & -1 & 0 & 0 & 0 & -1 & 0 & 0 & 0 & 0 & -1 & 0 & 0\\
0 & 0 & 0 & 0 & 0 & 1 & -1 & 0 & 0 & 0 & -1 & 0 & 1 & 0 & 0 & 0 & 1 & 0 & 0 & 0 & 0 & 0 & -1 & 0\\
0 & 0 & 0 & 0 & 0 & 0 & 0 & 0 & 1 & 0 & 0 & -1 & 0 & 0 & 0 & 0 & 0 & 0 & 0 & 0 & 1 & 0 & 0 & -1
\end{array}\right)\]
\normalsize
and the rank of this matrix is $15$.
Therefore, we conclude that
\begin{align} \label{HH1}
\begin{split}
&\dim_k \HH^0(\La_\e) =1,\\
&\dim_k  \HH^1(\La_\e)=6,\\
&\dim_k  \HH^2(\La_\e)= 9,\\ 
&\dim_k \HH^{i}(\La_\e)=0 \;\; (i\geq 3). 
\end{split}
\end{align}

\begin{rem}
The above algebraic computation allows us to compare this case and the next case.
On the other hand, in this case, we have $\Db(\mod \La_\e) \cong \Db(\qgr A_\e) \cong  \Db(\coh \P^1\times \P^1)$, so it is possible to compute the Hochschild cohomology in a sheaf-theoretical way.
In fact, it follows that  
\begin{align*}
&\dim_k \HH^0(\P^1\times \P^1) =1, \\
&\dim_k  \HH^1(\P^1\times \P^1)=6,\\
&\dim_k  \HH^2(\P^1\times \P^1)= 9, \\
&\dim_k \HH^{i}(\P^1\times \P^1)=0 \;\; (i\geq 3). 
\end{align*}
by Hochschild-Kostant-Rosenberg theorem (\cite[Corollary 0.6]{Y}). See \cite[Example 6]{Be} for example.
\end{rem} 

\noindent \textbf{The case \textrm{ (QS2)}.}
Next, let us discuss the case (QS2).
In this case, $A_\e =S_\e/(f_\e)$ where
$S_\e=k\<x, y, z, w\>/(xy+yx, xz+zx, xw-wx, zy+yz, wy+yw, zw+wz)$
and $f_\e= x^2+y^2+z^2+w^2$.
Since $\mnull_{\F_2} \De_{\e}=1$, we have $\al =|\M_\e|=2$
and therefore $\uCM(A) \cong \Db(\mod k^2)$.
It is easy to see that
\[
\left(
\begin{pmatrix}
x+iw &y+z \\
y+z &x-iw
\end{pmatrix},
\begin{pmatrix}
x-iw &y+z \\
y+z &x+iw
\end{pmatrix}
\right)
\;\text{and}\;
\left(
\begin{pmatrix}
x+iw &y-z \\
y-z &x-iw
\end{pmatrix},
\begin{pmatrix}
x-iw &y-z \\
y-z &x+iw
\end{pmatrix}
\right)
\]
induce noncommutative matrix factorizations of $f_\e$ in $S_\e$  where $i:=\sqrt{-1}$
 (see \cite[Theorem 4.4 (2)]{MUm}).
Thus 
\[X_1 := \Coker \begin{pmatrix}
x+iw &y+z \\
y+z &x-iw
\end{pmatrix}\cdot,\;
X_2 := \Coker \begin{pmatrix}
x+iw &y-z \\
y-z &x-iw
\end{pmatrix}\cdot \;
\in \M_\e \]
(see \cite[Theorem 6.5]{MUm}).
Hence $\La_\e$ is given by the quiver
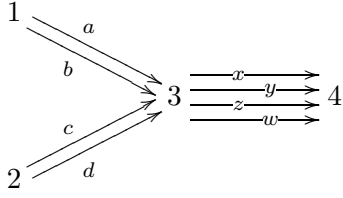
\begin{align*}
\xymatrix @R=1.5pc@C=4pc{
1 \ar@<0.5ex>[rd]^(0.4){a} \ar@<-0.5ex>[rd]_(0.4){b} \\
&3
\ar@<-1.8ex>[r]|(0.6){w}
\ar@<-0.6ex>[r]|(0.4){z}
\ar@<0.6ex>[r]|(0.6){y}
\ar@<1.8ex>[r]|(0.4){x}
&4 \\
2 \ar@<0.5ex>[ru]^(0.4){c} \ar@<-0.5ex>[ru]_(0.4){d}
}
\end{align*}
with relations
\begin{align*}
&g_1:=ax+iaw+by+bz=0,
&&g_2:=ay+az+bx-ibw=0,\\
&g_3:=cx+icw+dy-dz=0,
&&g_4:=cy-cz+dx-idw=0.
\end{align*}
Since $\gldim \La_\e= 2$,
we can construct the Hochschild complex (\ref{hcpx}).
Since the quiver is connected and has no oriented cycles,
it follows that $\dim_k \HH^0(\La_\e) =1$.
One can verify that $\dim_k  (P^{0})^*=4, \dim_k  (P^{1})^*=24, \dim_k  (P^{2})^*=24$ and $\dim_k \Im (A^2)^*=20$.
In fact, $(P^{0})^*$ has a basis $\{\th_{e_1},\dots,\th_{e_4}\}$,
$(P^{1})^*$ has an ordered basis
\begin{align*}\Theta^1 = 
(&\th_{a,a},\th_{b,a},\th_{a,b},\th_{b,b},
\th_{c,c},\th_{d,c},\th_{c,d},\th_{d,d},
\th_{x,x},\th_{y,x},\th_{z,x},\th_{w,x},\\
&\th_{x,y},\th_{y,y},\th_{z,y},\th_{w,y},
\th_{x,z},\th_{y,z},\th_{z,z},\th_{w,z},
\th_{x,w},\th_{y,w},\th_{z,w},\th_{w,w}
),
\end{align*}
and $(P^{2})^*$ has an ordered basis 
\begin{align*}\Theta^2 = 
(&\th_{ay,g_1},\th_{az,g_1},\th_{iaw,g_1},
\th_{by,g_1},\th_{bz,g_1},\th_{ibw,g_1},
\th_{ay,g_2},\th_{az,g_2},\th_{iaw,g_2},
\th_{by,g_2},\th_{bz,g_2},\th_{ibw,g_2},\\
&\th_{cy,g_3},\th_{cz,g_3},\th_{icw,g_3},
\th_{dy,g_3},\th_{dz,g_3},\th_{idw,g_3},
\th_{cy,g_4},\th_{cz,g_4},\th_{icw,g_4},
\th_{dy,g_4},\th_{dz,g_4},\th_{idw,g_4}
).
\end{align*}
Moreover, one can verify that the matrix representation of $(A^2)^*$ with respect to $\Theta^1$ and $\Theta^2$ is
\footnotesize
\[\scriptstyle \left(\begin{array}{cccccccccccccccccccccccc} 
0 & -1 & 1 & 0 & 0 & 0 & 0 & 0 & 0 & 1 & 0 & 0 & -1 & 0 & 0 & 0 & -1 & 0 & 0 & 0 & 0 & 1 & 0 & 0\\
0 & -1 & 1 & 0 & 0 & 0 & 0 & 0 & 0 & 0 & 1 & 0 & -1 & 0 & 0 & 0 & -1 & 0 & 0 & 0 & 0 & 0 & 1 & 0\\
0 & 0 & 0 & 0 & 0 & 0 & 0 & 0 & -1 & 0 & 0 & 1 & 0 & 0 & 0 & 0 & 0 & 0 & 0 & 0 & -1 & 0 & 0 & 1\\
-1 & 0 & 0 & 1 & 0 & 0 & 0 & 0 & -1 & 0 & 0 & 0 & 0 & 1 & 0 & 0 & 0 & 1 & 0 & 0 & -1 & 0 & 0 & 0\\
-1 & 0 & 0 & 1 & 0 & 0 & 0 & 0 & -1 & 0 & 0 & 0 & 0 & 0 & 1 & 0 & 0 & 0 & 1 & 0 & -1 & 0 & 0 & 0\\
0 & 2 & 0 & 0 & 0 & 0 & 0 & 0 & 0 & 0 & 0 & 0 & 1 & 0 & 0 & 1 & 1 & 0 & 0 & 1 & 0 & 0 & 0 & 0\\
1 & 0 & 0 & -1 & 0 & 0 & 0 & 0 & -1 & 0 & 0 & 0 & 0 & 1 & 0 & 0 & 0 & 1 & 0 & 0 & 1 & 0 & 0 & 0\\
1 & 0 & 0 & -1 & 0 & 0 & 0 & 0 & -1 & 0 & 0 & 0 & 0 & 0 & 1 & 0 & 0 & 0 & 1 & 0 & 1 & 0 & 0 & 0\\
0 & 0 & -2 & 0 & 0 & 0 & 0 & 0 & 0 & 0 & 0 & 0 & -1 & 0 & 0 & 1 & -1 & 0 & 0 & 1 & 0 & 0 & 0 & 0\\
0 & 1 & -1 & 0 & 0 & 0 & 0 & 0 & 0 & 1 & 0 & 0 & -1 & 0 & 0 & 0 & -1 & 0 & 0 & 0 & 0 & -1 & 0 & 0\\
0 & 1 & -1 & 0 & 0 & 0 & 0 & 0 & 0 & 0 & 1 & 0 & -1 & 0 & 0 & 0 & -1 & 0 & 0 & 0 & 0 & 0 & -1 & 0\\
0 & 0 & 0 & 0 & 0 & 0 & 0 & 0 & 1 & 0 & 0 & 1 & 0 & 0 & 0 & 0 & 0 & 0 & 0 & 0 & -1 & 0 & 0 & -1\\
0 & 0 & 0 & 0 & 0 & -1 & 1 & 0 & 0 & 1 & 0 & 0 & -1 & 0 & 0 & 0 & 1 & 0 & 0 & 0 & 0 & 1 & 0 & 0\\
0 & 0 & 0 & 0 & 0 & 1 & -1 & 0 & 0 & 0 & 1 & 0 & 1 & 0 & 0 & 0 & -1 & 0 & 0 & 0 & 0 & 0 & 1 & 0\\
0 & 0 & 0 & 0 & 0 & 0 & 0 & 0 & -1 & 0 & 0 & 1 & 0 & 0 & 0 & 0 & 0 & 0 & 0 & 0 & -1 & 0 & 0 & 1\\
0 & 0 & 0 & 0 & -1 & 0 & 0 & 1 & -1 & 0 & 0 & 0 & 0 & 1 & 0 & 0 & 0 & -1 & 0 & 0 & -1 & 0 & 0 & 0\\
0 & 0 & 0 & 0 & 1 & 0 & 0 & -1 & 1 & 0 & 0 & 0 & 0 & 0 & 1 & 0 & 0 & 0 & -1 & 0 & 1 & 0 & 0 & 0\\
0 & 0 & 0 & 0 & 0 & 2 & 0 & 0 & 0 & 0 & 0 & 0 & 1 & 0 & 0 & 1 & -1 & 0 & 0 & -1 & 0 & 0 & 0 & 0\\
0 & 0 & 0 & 0 & 1 & 0 & 0 & -1 & -1 & 0 & 0 & 0 & 0 & 1 & 0 & 0 & 0 & -1 & 0 & 0 & 1 & 0 & 0 & 0\\
0 & 0 & 0 & 0 & -1 & 0 & 0 & 1 & 1 & 0 & 0 & 0 & 0 & 0 & 1 & 0 & 0 & 0 & -1 & 0 & -1 & 0 & 0 & 0\\
0 & 0 & 0 & 0 & 0 & 0 & -2 & 0 & 0 & 0 & 0 & 0 & -1 & 0 & 0 & 1 & 1 & 0 & 0 & -1 & 0 & 0 & 0 & 0\\
0 & 0 & 0 & 0 & 0 & 1 & -1 & 0 & 0 & 1 & 0 & 0 & -1 & 0 & 0 & 0 & 1 & 0 & 0 & 0 & 0 & -1 & 0 & 0\\
0 & 0 & 0 & 0 & 0 & -1 & 1 & 0 & 0 & 0 & 1 & 0 & 1 & 0 & 0 & 0 & -1 & 0 & 0 & 0 & 0 & 0 & -1 & 0\\
0 & 0 & 0 & 0 & 0 & 0 & 0 & 0 & 1 & 0 & 0 & 1 & 0 & 0 & 0 & 0 & 0 & 0 & 0 & 0 & -1 & 0 & 0 & -1
\end{array}\right)\]
\normalsize
and the rank of this matrix is $20$.
Therefore, we conclude that
\begin{align} \label{HH2}
\begin{split}
&\dim_k \HH^0(\La_\e) =1,\\
&\dim_k  \HH^1(\La_\e)=1,\\
&\dim_k  \HH^2(\La_\e)= 4,\\
&\dim_k \HH^{i}(\La_\e)=0 \;\; (i\geq 3). 
\end{split}
\end{align}

\noindent \textbf{The case \textrm{ (QS3)}.}
Let us focus on the case (QS3).
In this case, $A_\e =S_\e/(f_\e)$ where
$S_\e=k\<x, y, z, w\>/(xy+yx, xz+zx, xw+wx, zy+yz, wy+yw, zw+wz)$
and $f_\e= x^2+y^2+z^2+w^2$.
By \exref{ex.-1}, 
$\La_\e$ is given by the quiver
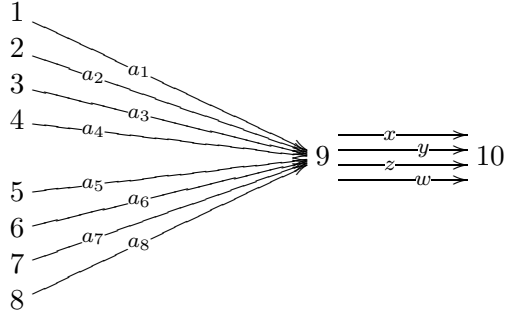
\begin{align*}
\xymatrix @R=0pc@C=4pc{
1 \ar[rrdddd]|(0.4){a_{1}}\\
2 \ar[rrddd]|(0.25){a_{2}} \\
3 \ar[rrdd]|(0.4){a_{3}}\\
4 \ar[rrd]|(0.25){a_{4}}\\
&&9
\ar@<-1.8ex>[r]|(0.6){w}
\ar@<-0.6ex>[r]|(0.4){z}
\ar@<0.6ex>[r]|(0.6){y}
\ar@<1.8ex>[r]|(0.4){x}
&10 \\
5 \ar[rru]|(0.25){a_{5}}\\
6 \ar[rruu]|(0.4){a_{6}}\\
7 \ar[rruuu]|(0.25){a_{7}}\\
8 \ar[rruuuu]|(0.4){a_{8}}
}
\end{align*}
with relations
\begin{align*}
&g_{1}:=a_{1}x +a_{1}y +a_{1}z +a_{1}w=0, 
&&g_{2}:=a_{2}x +a_{2}y +a_{2}z -a_{2}w=0,\\
&g_{3}:=a_{3}x +a_{3}y -a_{3}z +a_{3}w=0, 
&&g_{4}:=a_{4}x +a_{4}y -a_{4}z -a_{4}w=0,\\
&g_{5}:=a_{5}x -a_{5}y +a_{5}z +a_{5}w=0, 
&&g_{6}:=a_{6}x -a_{6}y +a_{6}z -a_{6}w=0,\\
&g_{7}:=a_{7}x -a_{7}y -a_{7}z +a_{7}w=0, 
&&g_{8}:=a_{8}x -a_{8}y -a_{8}z -a_{8}w=0.
\end{align*}
Since $\gldim \La_\e= 2$,
we can construct the Hochschild complex (\ref{hcpx}).
Since the quiver is connected and has no oriented cycles,
it follows that $\dim_k \HH^0(\La_\e) =1$.
One can verify that $\dim_k  (P^{0})^*=10, \dim_k  (P^{1})^*=24, \dim_k  (P^{2})^*=24$ and $\dim_k \Im (A^2)^*=15$.
In fact, $(P^{0})^*$ has a basis $\{\th_{e_1},\dots,\th_{e_{10}}\}$,
$(P^{1})^*$ has an ordered basis
\begin{align*}\Theta^1 = 
(&\th_{a_1,a_1},\th_{a_2,a_2},\th_{a_3,a_3},\th_{a_4,a_4},
\th_{a_5,a_5},\th_{a_6,a_6},\th_{a_7,a_7},\th_{a_8,a_8},
\th_{x,x},\th_{y,x},\th_{z,x},\th_{w,x},\\
&\th_{x,y},\th_{y,y},\th_{z,y},\th_{w,y},
\th_{x,z},\th_{y,z},\th_{z,z},\th_{w,z},
\th_{x,w},\th_{y,w},\th_{z,w},\th_{w,w}
),
\end{align*}
and $(P^{2})^*$ has an ordered basis 
\begin{align*}\Theta^2 = 
(&\th_{a_1y,g_1},\th_{a_1z,g_1},\th_{a_1w,g_1},
\th_{a_2y,g_2},\th_{a_2z,g_2},\th_{a_2w,g_2},
\th_{a_3y,g_3},\th_{a_3z,g_3},\th_{a_3w,g_3},
\th_{a_4y,g_4},\th_{a_4z,g_4},\th_{a_4w,g_4},\\
&\th_{a_5y,g_5},\th_{a_5z,g_5},\th_{a_5w,g_5},
\th_{a_6y,g_6},\th_{a_6z,g_6},\th_{a_6w,g_6},
\th_{a_7y,g_7},\th_{a_7z,g_7},\th_{a_7w,g_7},
\th_{a_8y,g_8},\th_{a_8z,g_8},\th_{a_8w,g_8}
).
\end{align*}
Moreover, one can verify that the matrix representation of $(A^2)^*$ with respect to $\Theta^1$ and $\Theta^2$ is
\footnotesize
\[ \left(\begin{array}{cccccccccccccccccccccccc}
0&0&0&0&0&0&0&0&-1 & 1 & 0 & 0 & -1 & 1 & 0 & 0 & -1 & 1 & 0 & 0 & -1 & 1 & 0 & 0\\
0&0&0&0&0&0&0&0&-1 & 0 & 1 & 0 & -1 & 0 & 1 & 0 & -1 & 0 & 1 & 0 & -1 & 0 & 1 & 0\\
0&0&0&0&0&0&0&0&-1 & 0 & 0 & 1 & -1 & 0 & 0 & 1 & -1 & 0 & 0 & 1 & -1 & 0 & 0 & 1\\
0&0&0&0&0&0&0&0&-1 & 1 & 0 & 0 & -1 & 1 & 0 & 0 & -1 & 1 & 0 & 0 & 1 & -1 & 0 & 0\\
0&0&0&0&0&0&0&0&-1 & 0 & 1 & 0 & -1 & 0 & 1 & 0 & -1 & 0 & 1 & 0 & 1 & 0 & -1 & 0\\
0&0&0&0&0&0&0&0&1 & 0 & 0 & 1 & 1 & 0 & 0 & 1 & 1 & 0 & 0 & 1 & -1 & 0 & 0 & -1\\
0&0&0&0&0&0&0&0&-1 & 1 & 0 & 0 & -1 & 1 & 0 & 0 & 1 & -1 & 0 & 0 & -1 & 1 & 0 & 0\\
0&0&0&0&0&0&0&0&1 & 0 & 1 & 0 & 1 & 0 & 1 & 0 & -1 & 0 & -1 & 0 & 1 & 0 & 1 & 0\\
0&0&0&0&0&0&0&0&-1 & 0 & 0 & 1 & -1 & 0 & 0 & 1 & 1 & 0 & 0 & -1 & -1 & 0 & 0 & 1\\
0&0&0&0&0&0&0&0&-1 & 1 & 0 & 0 & -1 & 1 & 0 & 0 & 1 & -1 & 0 & 0 & 1 & -1 & 0 & 0\\
0&0&0&0&0&0&0&0&1 & 0 & 1 & 0 & 1 & 0 & 1 & 0 & -1 & 0 & -1 & 0 & -1 & 0 & -1 & 0\\
0&0&0&0&0&0&0&0&1 & 0 & 0 & 1 & 1 & 0 & 0 & 1 & -1 & 0 & 0 & -1 & -1 & 0 & 0 & -1\\
0&0&0&0&0&0&0&0&1 & 1 & 0 & 0 & -1 & -1 & 0 & 0 & 1 & 1 & 0 & 0 & 1 & 1 & 0 & 0\\
0&0&0&0&0&0&0&0&-1 & 0 & 1 & 0 & 1 & 0 & -1 & 0 & -1 & 0 & 1 & 0 & -1 & 0 & 1 & 0\\
0&0&0&0&0&0&0&0&-1 & 0 & 0 & 1 & 1 & 0 & 0 & -1 & -1 & 0 & 0 & 1 & -1 & 0 & 0 & 1\\
0&0&0&0&0&0&0&0&1 & 1 & 0 & 0 & -1 & -1 & 0 & 0 & 1 & 1 & 0 & 0 & -1 & -1 & 0 & 0\\
0&0&0&0&0&0&0&0&-1 & 0 & 1 & 0 & 1 & 0 & -1 & 0 & -1 & 0 & 1 & 0 & 1 & 0 & -1 & 0\\
0&0&0&0&0&0&0&0&1 & 0 & 0 & 1 & -1 & 0 & 0 & -1 & 1 & 0 & 0 & 1 & -1 & 0 & 0 & -1\\
0&0&0&0&0&0&0&0&1 & 1 & 0 & 0 & -1 & -1 & 0 & 0 & -1 & -1 & 0 & 0 & 1 & 1 & 0 & 0\\
0&0&0&0&0&0&0&0&1 & 0 & 1 & 0 & -1 & 0 & -1 & 0 & -1 & 0 & -1 & 0 & 1 & 0 & 1 & 0\\
0&0&0&0&0&0&0&0&-1 & 0 & 0 & 1 & 1 & 0 & 0 & -1 & 1 & 0 & 0 & -1 & -1 & 0 & 0 & 1\\
0&0&0&0&0&0&0&0&1 & 1 & 0 & 0 & -1 & -1 & 0 & 0 & -1 & -1 & 0 & 0 & -1 & -1 & 0 & 0\\
0&0&0&0&0&0&0&0&1 & 0 & 1 & 0 & -1 & 0 & -1 & 0 & -1 & 0 & -1 & 0 & -1 & 0 & -1 & 0\\
0&0&0&0&0&0&0&0&1 & 0 & 0 & 1 & -1 & 0 & 0 & -1 & -1 & 0 & 0 & -1 & -1 & 0 & 0 & -1
\end{array}\right)\]
\normalsize
and the rank of this matrix is $15$.
Therefore, we conclude that
\begin{align} \label{HH3}
\begin{split}
&\dim_k \HH^0(\La_\e) =1, \\
&\dim_k  \HH^1(\La_\e)=0,\\
&\dim_k  \HH^2(\La_\e)= 9, \\
&\dim_k \HH^{i}(\La_\e)=0 \;\; (i\geq 3). 
\end{split}
\end{align}

The results of the calculations are summarized in the following table.

\begin{table}[h] 
\centering
\begin{tabular}{lccc}  
\toprule
&(QS1) & (QS2)  &(QS3)  \\
\midrule
$\dim_k (P^{0})^*$&4&4&10 \\
$\dim_k (P^{1})^*$&24&24&24 \\
$\dim_k (P^{2})^*$&24&24&24 \\
$\dim_k \Im(A^{1})^*$&3 &3 &9 \\
$\dim_k \Im(A^{2})^*$&15&20&15 \\
$\dim_k \HH^0(\La_\e)$&1&1&1 \\
$\dim_k \HH^1(\La_\e)$&6&1&0 \\
$\dim_k \HH^2(\La_\e)$&9&4&9 \\
$\dim_k \HH^i(\La_\e)\; (i\geq 3)$&0&0&0  \\
\bottomrule
\end{tabular}
\caption{Calculation results on Hochschild cohomology of $\La_\e$}  
\label{tab.HH}
\end{table}
Combining  \propref{prop.GM}, \lemref{lem.HH}, and Table \ref{tab.HH}, 
we reach the following conclusion.

\begin{thm} 
Assume that  $n=4$. 
\begin{enumerate}
\item For any $A_\e$, the derived category $\Db(\qgr A_\e)$ is equivalent to exactly one of the following categories:
\begin{itemize}
\item $\Db(\qgr k[x, y, z, w]/(x^2+y^2+z^2+w^2)) \cong \Db(\coh \P^1\times \P^1)$;
\item $\Db(\qgr k\<x, y, z, w\>/(
xy+yx,\ xz+zx,\ xw-wx,\ zy+yz,\ wy+yw,\ zw+wz,\ x^2+y^2+z^2+w^2))$;
\item $\Db(\qgr k\<x, y, z, w\>/(
xy+yx,\ xz+zx,\ xw+wx,\ zy+yz,\ wy+yw,\ zw+wz,\ x^2+y^2+z^2+w^2))$.
\end{itemize}
\item For any $A_\e$ and $A_{\e'}$, the following are equivalent.
\begin{enumerate}
\item $G_{\e'}$ is obtained from $G_{\e}$ by applying finitely many mutations up to isomorphism.
\item $\GrMod A_{\e} \cong \GrMod A_{\e'}$.
\item $\qgr A_{\e} \cong \qgr A_{\e'}$.
\item $\Db(\qgr A_{\e})\cong \Db(\qgr A_{\e'})$.
\end{enumerate}
\end{enumerate}
\end{thm}

\section*{Acknowledgments}
The author would like to thank Izuru Mori and Osamu Iyama for valuable comments and suggestions.
The author is also grateful to the referee for various comments that helped improve the exposition of the paper.

\end{document}